\documentclass[a4paper,11pt,DIV=11,oneside,headings=small,bibliography=totoc,abstract=on]{scrartcl}
\usepackage{amssymb,amsmath,amsfonts,amsthm,url,mathtools,microtype}
\allowdisplaybreaks
\usepackage[english]{babel}
\usepackage[T1]{fontenc}
\usepackage[utf8]{inputenc}
\usepackage{enumerate,expdlist}
\usepackage{tikz,pgf,float,booktabs,multirow,ifthen}
\usepackage{array}
\usetikzlibrary{patterns}
\usepackage[noblocks]{authblk}

\renewcommand{\epsilon}{\varepsilon}
\newcommand{\euler}{\mathrm{e}}
\newcommand{\Eins}{\mathbf{1}}
\newcommand{\drm}{\mathrm{d}}
\newcommand{\RR}{\mathbb{R}}
\newcommand{\CC}{\mathbb{C}}
\newcommand{\NN}{\mathbb{N}}	
\newcommand{\ZZ}{\mathbb{Z}}
\newcommand{\cB}{\mathcal{B}}
\newcommand{\cF}{\mathcal{F}}
\newcommand{\cH}{\mathcal{H}}
\newcommand{\cL}{\mathcal{L}}
\newcommand{\cX}{{\mathcal{X}}}
\newcommand{\cY}{{\mathcal{Y}}}
\newcommand{\cZ}{{\mathcal{Z}}}
\newcommand{\supp}{\operatorname{supp}}
\newcommand{\ran}{\operatorname{Ran}}
\newcommand{\Ker}{\operatorname{Ker}}
\newcommand{\obs}{\mathrm{obs}}
\newcommand{\const}{C_{\mathrm{obs}}}
\newcommand{\cost}{C_T}
\newcommand{\si}{{\mathrm{si}}}
\newcommand{\esssup}{\operatornamewithlimits{ess\,sup}}
\newcommand{\infspec}{\kappa}
\newtheorem{theorem}{Theorem}[section]
\newtheorem{lemma}[theorem]{Lemma}
\newtheorem{corollary}[theorem]{Corollary}
\theoremstyle{definition}
\newtheorem{definition}[theorem]{Definition}
\newtheorem{example}[theorem]{Example}
\theoremstyle{remark}
\newtheorem{remark}[theorem]{Remark}
\usepackage[]{hyperref}
	\usepackage{color}
 	\definecolor{darkred}{rgb}{0.5,0,0}
 	\definecolor{darkgreen}{rgb}{0,0.5,0}
 	\definecolor{darkblue}{rgb}{0,0,0.5}  	\hypersetup{colorlinks,linkcolor=darkblue,filecolor=darkgreen,urlcolor=darkred,citecolor=darkblue}

\begin{document}
%
% %%%%%%%%%%%%%%%%%%%%%%%%%%%%%%%%%%%%%%%%%%%%%%%%%%%%%%%%%%%%%%%%%%%%%%
% %---------------------------------------------------------------------
%
%         T I T L E
%
% %---------------------------------------------------------------------
% %%%%%%%%%%%%%%%%%%%%%%%%%%%%%%%%%%%%%%%%%%%%%%%%%%%%%%%%%%%%%%%%%%%%%%
% %
%
\title{Sharp estimates and homogenization of the control cost of the heat equation on large domains}
\author{Ivica Naki\'c}
\affil{University of Zagreb, Faculty of Science, Department of Mathematics, Zagreb, Croatia}
\author{Matthias T\"aufer}
\affil{Queen Mary University of London, School of Mathematical Sciences, London, United Kingdom}
\author{Martin Tautenhahn}
\affil{Technische Universit\"at Chemnitz, Fakult\"at f\"ur Mathematik, Chemnitz, Germany}
\author{Ivan Veseli\'c}
\affil{Technische Universit\"at Dortmund, Fakult\"at f\"ur Mathematik, Dortmund, Germany}
\date{\vspace{-3em}}
\maketitle
\begin{abstract}
We prove new bounds on the control cost for the abstract heat equation,
assuming a spectral inequality or uncertainty relation for spectral projectors.
In particular, we specify quantitatively  how upper bounds on the control cost depend on the constants in the spectral inequality.
This is then applied to the heat flow on bounded and unbounded domains modeled by a Schr\"odinger semigroup.
This means that the heat evolution generator is allowed to contain a potential term.
The observability/control set is assumed to obey an equidistribution or a thickness condition, depending on the context.
Complementary lower bounds and examples show that our control cost estimates are sharp in certain asymptotic regimes.
One of these is dubbed homogenization regime and corresponds to the situation
where the control set becomes more and more evenly distributed throughout the domain while its density remains constant.
\\[1ex]
\\[1ex]
\end{abstract}
\section{Introduction}
Let us start by describing the most important example which has motivated our study of control cost estimates for the heat equation.
Consider the inhomogeneous heat equation with heat generation term $ - V$ on suitable domains $\Omega \subset \RR^d$ (in particular, $\Omega = \RR^d$ is allowed as well) given by
\begin{equation}
 \label{eq:heat_equation_introduction}
 \dot w + (- \Delta + V) w = \Eins_{S} u,
 \quad
 w(0) = w_0 \in L^2(\Omega) ,
\end{equation}
where $w,u \in L^2([0,T] \times \Omega)$, $V \in L^\infty(\Omega)$, and where the \emph{control set} $S \subset \Omega$ is measurable with a positive measure.
Hence the influence of the \emph{control function $u$} is restricted to the set $S$.
The system \eqref{eq:heat_equation_introduction} is \emph{null-controllable} if for every $w_0 \in L^2(\Omega)$ there exists a \emph{control function} $u = u_{w_0} \in L^2([0,T] \times \Omega)$ such that the solution of~\eqref{eq:heat_equation_introduction} satisfies $w(T) = 0$, cf.~\cite{TucsnakW-09}.
In this case, the control cost in time $T$ is the least constant $C_T$ such that
 \[
 \lVert u_{w_0} \rVert_{L^2([0,T] \times \Omega)} \leq C_T \lVert w_0 \rVert_{L^2(\Omega)}
 \]
holds for all $w_0 \in L^2(\Omega)$.
\par
The aim of this paper is to investigate sharp upper and lower bounds on the control cost in time $T > 0$ of the controlled heat equation \eqref{eq:heat_equation_introduction},
in particular, its dependence on the geometry of $S$.
This is a natural aim since it has been shown recently that in the case $\Omega=\RR^d$,
if the system is null controllable,
$S$ \emph{necessarily} has to satisfy certain geometric conditions, \cite{WangWZZ-17-arxiv,EgidiV-18}.
\par
We are able to establish the optimality of our bounds in certain asymptotic regimes.
A particularly appealing geometric regime is the homogenization scenario, in  which the control set $S \subset \Omega$ becomes more and more evenly distributed over $\Omega$ while keeping an overall lower bound on the relative density.
This corresponds to reducing local fluctuations in the density of the control set $S$. In such a homogenization regime we study the asymptotic behavior of the upper bound of the control cost.
\par
Note that in the context of control theory homogenization scenarios have been studied before, see e.g.~\cite{Zuazua-94,LopezZ-02}.
There however, as in classical homogenization theory, it is the differential operator generating the semigroup which is being homogenized,
rather than the observability set.
\par
So far, much more attention has been devoted to identifying the dependence of the control cost on the time parameter than to its geometric counterparts.
In \cite{Seidman-84} Seidman proved that for one-dimensional controlled heat systems the control cost in small time regime blows up at most exponentially.
This result was extended to arbitrary dimension by Fursikov and Imanuilov in \cite{FursikovI-96}. That the exponential blowup indeed occurs was established by G\"uichal \cite{Guichal-85} for one-dimensional systems and by Miller \cite{Miller-04} in the general case. Since then the bounds on the control cost have received a lot of attention
\cite{FernandezZ-00,Miller-04b,Phung-04,Miller-06b,Miller-06,MicuZ-06,TenenbaumT-07,DuyckaertsZZ-08,Miller-10,TenenbaumT-11,ErvedozaZ-11,LeRousseauL-12,Lissy-12,Lissy-15,BardosP-17,DardeE-19,EgidiV-18,LaurentL-18-arxiv,Phung-18}.
Most of the results were obtained for bounded domains, but recently unbounded domains also became a focus of interest
\cite{BeauchardP-18,EgidiV-17-arxiv,WangWZZ-17-arxiv,EgidiV-18,Egidi-18,EgidiNSTTV-20}.
Note that there has been also interest in
observability estimates if the measurement occurs only during a positive-measure subset of the time interval, see for instance \cite{WangZ-17}.
\par
The most common way to obtain a bound on the control cost is a \emph{final-state-observability estimate}
(an estimate concerning the dynamics of the corresponding adjoint system)
which in our context states that for all $T > 0$ there is a  $T$-dependent constant $C_\obs$ such that
\[
\lVert \euler^{(\Delta - V)T}\phi \rVert_{L^2(\Omega)}^2 \le C_\obs^2 \int_0^T \lVert \mathrm{e}^{(\Delta - V)t} \phi \rVert_{L^2(S)}^2 \drm t \quad \text{for all } \phi \in L^2(\Omega).
\]
Here $t \mapsto \euler^{(\Delta - V)t}$ denotes the $C_0$-semigroup generated by $\Delta - V$.
The duality between null-controllability and final-state-observability implies that $C_\obs$ is an upper bound for the control cost.
In the seminal papers \cite{LebeauR-95,LebeauZ-98,JerisonL-99} it has been
shown that one way to establish observability estimates
is to prove a spectral inequality (a particular type of uncertainty relation)
\[
\lVert \phi \rVert_{L^2(\Omega)} \le C \mathrm{e}^{c \sqrt{\lambda - \infspec}} \lVert \phi \rVert_{L^2(S)} \quad \text{for all } \lambda > \infspec \text{ and } \phi \in \ran P_{ - \Delta + V}(\lambda),
\]
where $C,c>0$ are constants, $P_{ - \Delta + V}(\lambda)$ is the  projector to the spectral subspace of $ - \Delta + V$ below $\lambda$, and $\infspec$ is the minimal spectral point of $ - \Delta + V$.
This is a particularly attractive technique since the spectral inequality does not involve the time variable, i.e.\ it concerns only the corresponding stationary system.
Consequently, quite a number of works developed abstract theorems to derive bounds on the control cost from spectral inequalities, each tailored for certain applications in mind.
Among them are \cite{Miller-10,TenenbaumT-11,BeauchardP-18}, which are also most closely related to our present paper.
In spite of the variety of such earlier results, none of them is sufficient for our purposes,
namely to provide sharp bounds on the control cost in several asymptotic scenarios
of interest to us.
\par
Hence, the first step to analyze heat control problems as described at the beginning of this section
was to establish null-controllability of an abstract parabolic system from a suitable spectral inequality,
together with an upper bound
on the control cost. This is spelled out in  Theorem~\ref{thm:obs_and_control}
whose proof
is inspired by the direct approach of~\cite{TenenbaumT-11}, since it turned out to be the one
which can be best generalized and optimized for the geometric situations we had in mind.
A simplified version of  Theorem~\ref{thm:obs_and_control}
states that if a non-negative operator $A$ on a Hilbert space $X$ satisfies
for all $\phi \in X$  and for all $\lambda > 0$
\begin{equation*}%\label{eq:UCP}
  \lVert P_A (\lambda) \phi \rVert_X^2
  \leq
  d_0 \euler^{d_1 \lambda^\gamma}
  \lVert B^\ast P_A (\lambda) \phi \rVert_U^2
 \end{equation*}
where $d_0 > 0$ , $d_1 \geq 0$ and $\gamma \in (0,1)$ are constants and $B\colon U \to X$ a bounded operator between Hilbert spaces, then for all $T > 0$ and all $\phi \in X$ we have the observability estimate
\begin{equation*}%\label{eq:obs}
 \bigl\lVert \euler^{-AT} \phi \bigr\rVert_X^2
 \leq
 \const^2
 \int_0^T  \bigl\lVert B^* \euler^{-At} \phi \bigr\rVert_U^2 \drm t ,
\end{equation*}
with
\[
  \const^2
  =
  \frac{C_1 d_0}{T}
  K_1^{C_2}
  \exp \left(  C_3
   \left( \frac{d_1}{T^\gamma} \right)^{\frac{1}{1 - \gamma}} \right)
     \quad \text{ and } \quad K_1 = 2 d_0  \lVert B \rVert_{\cL (U, X)}^2 + 1 .
\]
Theorem~\ref{cor:upper_bounds_control_cost} gives an extension of this result to the case where $A$ is merely lower semi-bounded, but not necessarily non-negative.
It is complemented by lower bounds, spelled out in Theorem~\ref{thm:lower_bound_control_cost},
and Table~\ref{table:asymptotics} comparing upper and lower bound is several asymptotic scenarios.
\par
 On the quantitative level, the key improvement over the existing results is  the dependence of our control cost estimate on the parameters coming from the spectral inequality, cf.\ Remark~\ref{rem:main_theorem}.
 While Theorems~\ref{thm:obs_and_control} and \ref{cor:upper_bounds_control_cost} also cover the case when the control operator $B$ is not bounded, hence enabling its use in the case of boundary control,
 we do not pursue this question in the present paper.
%%%%%%%%%%%%%%%%%%%%%%%%%%%%%%%%%%%%%%%%%%%%%%%%%%
%%%%%%%%%%%%%%%%%%%%%%%%%%%%%%%%%%%%%%%%%%%%%%%%%%
% \par
Let us also note that our paper has inspired a follow up work in the more general framework of control theory in Banach spaces~\cite{GallaunST}.
%The observability constant in \cite{GallaunST} is qualitatively and quantitatively of the same form as ours, and hence also allows to study homogenization limits.
%%%%%%%%%%%%%%%%%%%%%%%%%%%%%%%%%%%%%%%%%%%%%%%%%%
%%%%%%%%%%%%%%%%%%%%%%%%%%%%%%%%%%%%%%%%%%%%%%%%%%
\par
 Since our abstract theorem reduces the control cost estimate to a spectral inequality, it is also paramount for these spectral inequalities to have an explicit and -- if possible -- optimal dependence on parameters of interest.
 Recently, spectral inequalities with explicit geometry dependence on bounded and unbounded domains have been proved: in \cite{EgidiV-18,EgidiV-17-arxiv} for the free heat equation controlled by a thick set, and in \cite{NakicTTV-18+} for the heat equation with potential with  control supported on an equidistributed set.
 We combine these two spectral inequalities with our abstract theorem and obtain control cost estimates for the heat equation which are valid in a large class of (bounded and unbounded) domains
 and which depend on the control set $S$ via its geometric parameters.
 Our bounds are uniform in the heat generation term $V$ (they only depend on $\lVert V \rVert_{\infty}$) and are also uniform in $\Omega$ and $S$ in a certain sense.
 The obtained estimates are much more explicit than what existed before and, together with the uniformity in $\Omega$ and $S$ allow for the first time to study homogenization and de-homogenization limits.
\par
The paper is divided in three parts.
The core of Section~\ref{sec:abstract} are two theorems which spell out that spectral inequalities imply observability estimates.
The first one, Theorem~\ref{thm:obs_and_control}, concerns non-negative operators, the second one, Theorem~\ref{cor:upper_bounds_control_cost}, general
lower semi-bounded operators.
In the case of lower semi-bounded operators, the long time asymptotics of control cost depends on the growth bound of the corresponding semigroup.
In order to better understand the upper bounds proven in Theorems~\ref{thm:obs_and_control} and \ref{cor:upper_bounds_control_cost},
we compare them to lower bounds on the control of the heat equation for abstract systems and prove their sharpness in Theorem~\ref{thm:lower_bound_control_cost}.
In this section we also provide a thorough discussion of the lower and upper bounds of abstract control systems in Remark~\ref{rem:table}
and Table~\ref{table:asymptotics}. The following Section~\ref{sec:proofs} contains the proofs of Theorems~\ref{thm:obs_and_control} and  \ref{cor:upper_bounds_control_cost}
and a technical lemma, completing the first part of the paper.
\par
In Section~\ref{sec:applications} we then turn to system~\eqref{eq:heat_equation_introduction} and combine Theorem~\ref{cor:upper_bounds_control_cost} with the spectral inequalities from \cite{EgidiV-18} and \cite{NakicTTV-18+} to obtain bounds on the control cost for the free heat equation and free fractional heat equation controlled by a thick set, and a heat equation with a generation term controlled by a equidistributed set.
\par
Finally, Section~\ref{sec:homogenization} is devoted to studying homogenization and de-homogenization of the control cost.
\section{Abstract observability and null-controllability}
\label{sec:abstract}
For normed spaces $V$ and $W$ we denote by $\cL (V , W)$
the space of bounded linear operators from $V$ to $W$.
Let $X$ and $U$ be Hilbert spaces with inner products $\langle \cdot , \cdot \rangle$ and $\langle \cdot , \cdot \rangle_U$ and norms $\lVert \cdot \rVert$ and $\lVert \cdot \rVert_U$, respectively.
Let $A$ be a lower semibounded self-adjoint operator in $X$ with domain $\mathcal{D}(A)$.
We define $\infspec = \min \sigma(A)$ and denote by $\{P_H (\lambda) \colon \lambda \in \RR\}$ the resolution of identity of a self-adjoint operator $H$.
Next we will introduce the structure of a rigged Hilbert space or a Gelfand triple.
Let $\beta \in \RR$. On $X$ we define the scalar product and norm
\begin{equation}\label{eq:scalar}
 \langle x,y \rangle_{\beta}
 :=
 \bigl\langle (I + A^2)^{\beta / 2} x , (I + A^2)^{\beta / 2} y \bigr\rangle, \quad
 \Vert x\Vert_\beta:=\Vert (I + A^2)^{\beta / 2}x\Vert .
\end{equation}
For $\beta >0$ we denote by $X_{\beta}:=\mathcal{D}(I + A^2)^{\beta/2}$
the subset  of $X$ consisting of all elements $x$ with finite $\Vert x\Vert_\beta$ norm.
If we equip $X_\beta$ with the scalar product \eqref{eq:scalar} we obtain a Hilbert space, usually called the interpolation space.
For $\beta \leq 0$ we denote by $X_{\beta}\supset X$ the Hilbert space obtained as the completion of $X$ with respect to the norm $\Vert \cdot \Vert_\beta$ in \eqref{eq:scalar}, usually called the extrapolation space.
Note that $X_{-\beta}$ is the dual space of $X_{\beta}$ with respect to the pivot space $X$.
>From now on we assume that $\beta \le 0$ and $B \in \cL (U , X_{\beta})$.
Clearly, $X_0 = X$ and the case $\beta = 0$ is of particular interest for our applications in Sections~\ref{sec:applications} and~\ref{sec:homogenization}.
\par
For $T > 0$, we study the abstract inhomogeneous Cauchy problem
\begin{equation}
 \label{eq:abstract_control_system}
 \dot{w} + Aw = Bu,
 \quad
 w(0)=w_0 \in X ,
\end{equation}
where $w \in L^2([0,T], X)$ and $u \in L^2([0,T], U)$. The function $u$ is called \emph{control function}.
The mild solution of \eqref{eq:abstract_control_system} is given by
\begin{equation}
\label{eq:Duhamel_formula}
 w(t)
 =
 \euler^{- A t} w_0
 +
 \int_0^t \euler^{- A (t-s)} B u(s) \drm s,
 \quad
 t \in [0, T].
\end{equation}
Note that since $\ran B \subset X_\beta$, we need to give a meaning to the term $\euler^{-A(t - s)} B u(s)$.
For this purpose we introduce the semigroup  $S(t)$ in $X$ with generator $- A$ and use the symbol $\euler^{- A \cdot}$ for the unique extension of $S(t)$ to the space $X_\beta$.
More precisely, let $U_\beta \in \cL(X,X_\beta)$ be the isometric operator given as the unique extension of $(I + A^2)^{-\beta/2} \in \cL (X_{- \beta}, X)$.
Then we have $\euler^{- A t} = U_\beta S(t) {U_{\beta}}^{-1}$.
\par
\begin{remark} \label{lem:w_in_X}
Although we do not assume that $B$ is an admissible control operator (for the definition, see, for example \cite{TucsnakW-09}), we still have that
\begin{equation*}%\label{eq:w_in_X}
\text{for all}\ T >0 \ t\in [0,T] \ \text{the function defined in \eqref{eq:Duhamel_formula} satisfies} \ w(t) \in X.
\end{equation*}
This follows from
\[
 \euler^{- A t} X_\beta
 =
 U_\beta S(t) {U_{\beta}}^{-1} X_\beta
 =
 U_\beta S(t) X
 =
 U_\beta \mathcal{D}( A^\infty )
 \subset
 U_\beta X_{- \beta}
 =
 X
\]
for all $t > 0$ where $\mathcal{D}(A^\infty) = \bigcap_{n \in \NN} \mathcal{D}(A^n)$.
Here, the equality $U_\beta S(t) X = U_\beta \mathcal{D}( A^\infty )$ follows from the fact that $S(t)$ is an analytic semigroup, cf.\ \cite[Chapter IX.1.6]{Kato-95}.
This shows $\euler^{-A (t - s)} B u(s) \in X$ for all $s \in [0,t)$ which implies $w(t) \in X$.
\end{remark}
We now introduce two concepts, null-controllability and final-state-observability.
\begin{definition}
 The system~\eqref{eq:abstract_control_system} is \emph{null-controllable in time $T > 0$} if for every $w_0 \in X$ there exists a control function $u = u_{w_0} \in L^2([0,T], U)$ such that the solution \eqref{eq:Duhamel_formula} satisfies $w(T) = 0$. We call such a control function \emph{null-control function in time $T$}.
The \emph{input map in time $T$} is the bounded mapping $\cB^T \colon  L^2([0,T], U) \to X$ given by $\cB^T u = \int_0^T \euler^{- A (T-s)} B u(s) \drm s$.
\end{definition}
\begin{remark}
 Note that if the system~\eqref{eq:abstract_control_system} is null-controllable in time $T > 0$, then, by linearity of
 $\euler^{- A T}$, it is also controllable on the range of $\euler^{- A T}$. This means that for every $w_0 \in X$ and every
 $u_T \in \ran \euler^{- A T}$ there is a control function $u \in L^2([0,T], U)$ such that the solution
 of~\eqref{eq:abstract_control_system} satisfies $w(T) = u_T$.
\par
Taking into account~\eqref{eq:Duhamel_formula}, a null-control function $u$ in time $T$ satisfies $\euler^{- A T} u_0 + \cB^T u = 0$. Thus, the system ~\eqref{eq:abstract_control_system} is null-controllable in time $T > 0$
if and only if one has the relation $\ran \cB^T \supset \ran \euler^{- A T}$, which gives an alternative definition of
null-controllability in terms of the input map.
\end{remark}
We will always take duals of the spaces $X_\beta$ with respect to the pivot space $X$.
Hence, the dual of $X_\beta$ is $X_{- \beta}$ for all $\beta \in \RR$ and in particular $B^\ast \in \cL(X_{- \beta}, U)$.
In order to introduce the notion of final-state-observability, we consider the adjoint system
\begin{equation}
 \label{eq:abstract_control_system_free}
 \dot f + A f = 0, \quad
 y = B^\ast f , \quad
 f (0) = f_0 \in X ,
\end{equation}
where $f \in L^2 ([0,T],X)$.
\begin{definition}\label{def:final-state-obs}
 The system~\eqref{eq:abstract_control_system_free} is called \emph{final-state-observable} in time $T > 0$ if there is a constant
 $C_\obs > 0$ such that for all $f_0 \in X$ we have
 \begin{equation}
  \label{eq:abstract_observability_estimate}
  \lVert \euler^{- A T} f_0 \rVert_X^2
  \leq
  C_\obs^2
  \int_0^T
  \lVert B^\ast \euler^{- A t} f_0 \rVert_U^2
  \drm t .
 \end{equation}
 Ineq.~\eqref{eq:abstract_observability_estimate} is
 called~\emph{observability inequality}.
\end{definition}
By an analogous reasoning as above, we see that $f(t) = \euler^{- A t} f_0 \in X_{- \beta}$ for all $t > 0$ whence \eqref{eq:abstract_control_system_free} and the right hand side of~\eqref{eq:abstract_observability_estimate} are well-defined.
The following lemma, due to Douglas \cite{Douglas-66} and Dolecki and Russell \cite{DoleckiR-77}, puts these concepts into relation. For a proof we refer also to \cite{TucsnakW-09,TenenbaumT-11}.
\begin{lemma}
\label{lem:absop}
  Let $\cH_1, \cH_2,\cH_3$ be Hilbert spaces, and let $\cX \colon \cH_1 \to \cH_3$, $\cY \colon \cH_2 \to \cH_3$ be bounded operators.
  Then, the following are equivalent:
  \begin{enumerate}[(a)]
   \item $\ran \cX \subset \ran \cY$.
   \item There is $c>0$ such that $\lVert \cX^* z \rVert \le c \lVert \cY^* z \rVert $ for all $z\in \cH_3$.
   \item There is a bounded operator $\cZ \colon \cH_1 \to \cH_2$ satisfying $\cX = \cY \cZ$.
  \end{enumerate}
  Moreover, in this case, one has
  \begin{equation}\label{eq:absop}
   \inf\{ c\colon c\text{ as in (b)}\} = \inf\{\lVert \cZ \rVert  \colon \cZ \text{ as in (c)}\},
  \end{equation}
  and both infima are actually minima.
\end{lemma}
We note that (a) corresponds to null-controllability (with $\cX = \euler^{-AT}\colon X \to X$ )
and (b) corresponds to final-state-observability with $$\cY \colon L^2([0,T], U) \to X,\quad \cY u =\int_0^T \euler^{- A (T-s)} B(s) u(s) \drm s.$$
Lemma~\ref{lem:absop} provides another equivalent statement (c). It implies that there exists an operator $\cF\colon X \to L^2([0,T], U)$ such that $-\euler^{-A T} = \cB^T \cF$. Hence $\cF w_0$ provides a null-control function in time $T$.
Moreover, according to \eqref{eq:absop} the operator $\cF$ can be chosen with minimal norm.
\begin{remark}
The operator $\cF$ can even be chosen to be pointwise minimal.	
  Let $w_0 \in X$, $T > 0$, and $u$ be a null-control function in time $T$. Then the set of all null-control functions in time $T$ is a closed affine space of the form
  \[
   u + \Ker \cB^T.
  \]
  Let now $P$ denote the orthogonal projection onto $\Ker \cB^T$.
  Then we have $-\euler^{-AT} = \cB^T (I-P)\cF$ and the operator $(I- P) \cF$ does not depend on the choice of $\cF$.
  Moreover, it is easy to see that for every $w_0 \in X$, the function $(I - P) \cF w_0 \in L^2([0,T], U)$ is the unique control with minimal norm associated to the initial datum $w_0$.
  This implies in particular the second equality in \eqref{eq:definition_cost}.
\end{remark}
\begin{definition}
 Assume that the system~\eqref{eq:abstract_control_system} is null-controllable.
 We define the \emph{control cost in time $T$} as
 \begin{align} \label{eq:definition_cost}
  C_T
  &:=
  \sup_{\lVert w_0 \rVert = 1} \min \{ \lVert u \rVert_{L^2([0,T],U)} \colon \euler^{-A T} w_0 + \cB^T u = 0 \}
  =
  \min
  \{
  C_\obs \colon \text{$C_\obs$ satisfies }\eqref{eq:abstract_observability_estimate} \}.
 \end{align}
\end{definition}
Our first result concerns an observability inequality and hence null-controllability for an abstract parabolic system of the form~\eqref{eq:abstract_control_system_free}.
In the theorem, we assume a so-called spectral inequality, given in Ineq.~\eqref{eq:UCP}.

\begin{theorem}\label{thm:obs_and_control}
 Let $A $ be a non-negative, self-adjoint operator on $X$ and assume that there are $d_0 > 0$, $d_1 \geq 0$ and $\gamma \in (0,1)$ such that for all $\lambda > 0$ and all $\phi \in X$ we have
 \begin{equation}\label{eq:UCP}
  \lVert P_A (\lambda) \phi \rVert^2
  \leq
  d_0 \euler^{d_1 \lambda^\gamma}
  \lVert B^\ast P_A (\lambda) \phi \rVert_U^2.
 \end{equation}
Then for all $T > 0$ and all $\phi \in X$ we have the observability estimate
\begin{equation}\label{eq:obs}
 \bigl\lVert \euler^{-AT} \phi \bigr\rVert^2
 \leq
 \const^2
 \int_0^T  \bigl\lVert B^* \euler^{-At} \phi \bigr\rVert_U^2 \drm t ,
\end{equation}
where $\const$ satisfies
\[
  \const^2
  =
  \frac{C_1 d_0}{T}
  K_1^{C_2}
  \exp \left(  C_3
   \left( \frac{d_1 + (-\beta)^{C_4}}{T^\gamma} \right)^{\frac{1}{1 - \gamma}} \right)
     \quad \text{with} \quad K_1 = 2 d_0 \euler^{-\beta} \lVert B \rVert_{\cL (U, X_\beta)}^2 + 1 .
\]
Here, $C_i > 0$, $i \in \{1,2,3,4\}$, are constants depending only on $\gamma$.
They are explicitly given by Eq.~\eqref{eq:final_constant_lemma}.
Moreover, for all $T > 0$ the system~\eqref{eq:abstract_control_system} is null-controllable in time $T$ with cost satisfying $\cost \leq \const$.
\end{theorem}
 Note that the right hand side in \eqref{eq:UCP} is well-defined since $\ran P_A(\lambda)\subset X_{-\beta}$ for all $\lambda\in \mathbb{R}$, and \eqref{eq:obs} is well-defined as discussed below Definition~\ref{def:final-state-obs}. The mere statement that spectral inequalities imply observability estimates is not new.
The novel aspects of Theorem~\ref{thm:obs_and_control} are discussed in the following remark, while the proof is deferred to the next section.
\begin{remark}
\label{rem:main_theorem}
There exists a huge amount of earlier approaches which transfer spectral inequalities to observability inequalities, see e.g.\ \cite{LebeauR-95,Miller-10,TenenbaumT-11,LeRousseauL-12,BeauchardP-18}.
    Some of them are formulated in a more general setting and unlike our result do not require self-adjointness of $A$.
    However, the estimates on $\const^2$ therein are, with respect to the dependence on $d_0$ and $d_1$, not sufficient for our purpose in Section~\ref{sec:applications}. Let us explain this in more detail, and assume within this discussion that $\beta = 0$ and $\gamma = 1/2$.
    Our upper bound
    \[
    \const^2 = \frac{C_1 d_0}{T} K_1^{C_2} \exp \left( \frac{C_3 d_1^2}{T} \right)
    \]
    from Theorem~\ref{thm:obs_and_control} features the following properties:
 \begin{enumerate}[(i)]
  \item
  The exponent tends to zero if $d_1 \to 0$.
  \item
  The pre-factor $ C_1 d_0 K_1^{C_2} / T$ does not depend on $d_1$ and is proportional to $T^{-1}$.
  \item
  The estimate holds in a $d_1$-independent time interval (in our case $(0,\infty)$).
  \end{enumerate}
  All three properties are paramount for the applications to homogenization and
de-homo\-geniza\-tion in Section~\ref{sec:homogenization}.
  Let us stress that none of the earlier bounds we are aware of carry the features (i)--(iii) at the same time. For example, the papers \cite{Miller-10,BeauchardP-18} provide a bound of the form
  \begin{equation}\label{eq:BPS}
  \const^2 \leq C_1 \exp \left(\frac{C_2}{T}\right) ,
  \end{equation}
  where the dependence of the positive constants $C_1$ and $C_2$ on $d_0$ and $d_1$ can be inferred from their proof.
  Note that the bound~\eqref{eq:BPS} is missing the pre-factor $1/T$. Thus, $C_2$ in \eqref{eq:BPS} cannot be proportional to $d_1^2$ since for $d_1 = 0$ (full control) this contradicts the universal lower bound of order $1/T$, cf.\ Theorem~\ref{thm:lower_bound_control_cost}.
 \par
  In order to obtain our bound in Theorem~\ref{thm:obs_and_control}, we improve techniques developed in \cite{TenenbaumT-11}. Note that the bound given in \cite[Theorem~1.2]{TenenbaumT-11} already satisfies properties (i) and (iii), and carries the overall pre-factor $1/T$. However, it does not ensure that the influence of $d_1$ is confined only to the exponential term. Intricate parameter choices and estimates -- spelled out in Lemma~\ref{lem:Cobs} -- were necessary in order to achieve an estimate of the required form. Moreover, in contrast to \cite{TenenbaumT-11}, we do not require that the operator $A$ has discrete spectrum, thus extending the applicability e.g.\ to Schr\"odinger operators on unbounded subsets of $\RR^d$.
\end{remark}
\begin{remark}
The arguments used in the proof of Theorem~\ref{thm:obs_and_control} can be extended to the case where $B \colon [0,T] \to  \cL (U , X_{\beta})$ is time-dependent with only minimal modifications.
For the basic results about the integration theory on Hilbert spaces used here we refer to \cite{HilleP57} and \cite{DiestelU77}.
Let $\cH_i$ for $i \in\{1,2,3\}$ be separable Hilbert spaces and $I\subset \RR$ an interval.
Then
if $I \ni t \mapsto B(t)\in \cL(\cH_1, \cH_2)$ and
$I \ni t \mapsto A(t)\in \cL(\cH_2, \cH_3)$ are measurable
then the product $I \ni t \mapsto A(t)B(t)\in \cL(\cH_1, \cH_3)$
is measurable as well. If additionally
$I \ni t \mapsto x(t) \in \cH_1$ is measurable, then
$I \ni t \mapsto B(t)x(t) \in \cH_2$ is measurable as well.
Consequently the map $I \ni t \mapsto \lVert B(t)x \rVert^2= \langle x, B(t)^*B(t)x\rangle\in \RR$
is measurable too.
\par
In what follows let us assume that the Hilbert spaces $U$ and $X_\beta$
are separable, and that $B: [0,T] \to \cL(U,X_\beta)$ is measurable and uniformly bounded, meaning
 \begin{equation}\label{eq:sup_ess}
  \esssup_{t \in [0,T]} \lVert B(t) \rVert_{\cL(U, X_\beta)}
  <
  \infty.
 \end{equation}
Note that for any dense countable $U'\subset U\setminus \{0\}$ we have
\[
 \lVert B(t) \rVert_{\cL(U, X_\beta)}  =
 \sup_{u \in U'}  \lVert B(t) u \rVert_{X_\beta}/ \lVert u \rVert_U,
\]
 so that $t \to  \lVert B(t) \rVert_{\cL(U, X_\beta)}$ is measurable and the essential supremum (w.r.t. Lebesgue measure on $[0,T]$) makes sense.
Note that for every $t \in (0,T]$, the map $[0,t] \ni s \mapsto \euler^{- A (t - s)} $ is strongly continuous, hence measurable.
\par
An argument analogous to the discussion at the beginning of Section~\ref{sec:abstract} shows that we even have $\euler^{- A (T-s)} B(s) \allowbreak u(s) \in X$ for almost all $s \in [0,T]$.
In particular,
\[
  \cB^T f(t) = \int_0^t \euler^{- A (t-s)} B(s) u(s) \drm s \in X
\]
is well-defined for every $u \in L^2([0,T], U)$.
For each initial state $w_0\in X$ and every $u \in L^2([0,T], U)$, the evolution
\begin{equation*}
 w(t)
 =
 \euler^{- A t} w_0 + \cB^T u(t),
 \quad
 t \in [0, T] ,
\end{equation*}
solves the equation
\begin{equation*}
  \label{eq:time dependent_control_system}
 \dot w(t) + Aw(t) = B(t)u(t),
 \quad
 w(0)=w_0 \in X.
\end{equation*}
\end{remark}
The following theorem is the natural generalization of Theorem~\ref{thm:obs_and_control} to time-dependent $B$:
\begin{theorem}
 Let $\beta\leq 0$, $d_0 > 0$, $d_1 \geq 0$, $T > 0$, and $\gamma \in (0,1)$.
 Let $U, X_\beta$ be separable, $A \geq 0$,  and let $B \colon [0,T] \to  \cL (U , X_{\beta})$ be measurable and satisfy \eqref{eq:sup_ess}.
 Assume that
\begin{equation*}
  \lVert P_A (\lambda) \phi \rVert^2
  \leq
  d_0 \euler^{d_1 \lambda^\gamma}
  \lVert B(t)^\ast P_A (\lambda) \phi \rVert_U^2
  \quad
  \text{for almost all $t \in [0,T]$}.
 \end{equation*}
Then for  all $\phi \in X$ we have the observability estimate
\begin{equation*}
 \bigl\lVert \euler^{-AT} \phi \bigr\rVert^2
 \leq
 \const^2
 \int_0^T  \bigl\lVert B^* \euler^{-At} \phi \bigr\rVert_U^2 \drm t ,
\end{equation*}
where $\const$ satisfies
\begin{equation*}
  \const^2
  =
  \frac{C_1 d_0}{T}
  K_1^{C_2}
  \exp \left(  C_3
   \left( \frac{d_1 + (-\beta)^{C_4}}{T^\gamma} \right)^{\frac{1}{1 - \gamma}} \right)
  \quad
\end{equation*}
with
\begin{equation*}
 K_1 =
 1+\operatorname{ess\ sup}_{t\in[0,T]} \left(2d_0\euler^{-\beta} \lVert B(t)\rVert_{\cL(U,X_\beta)}^2\right).
\end{equation*}
Here, $C_i > 0$, $i \in \{1,2,3,4\}$, are constants depending only on $\gamma$.
They are explicitly given by Eq.~\eqref{eq:final_constant_lemma}.
Moreover, the system~\eqref{eq:time dependent_control_system} is null-controllable in time $T$ with cost satisfying $\cost \leq \const$.
\end{theorem}
\begin{proof}
 The observability estimate is proved by following verbatim the proof of Theorem~\ref{thm:obs_and_control}.
To prove null-controllability and the control cost bound one uses part (b) of Lemma~\ref{lem:absop} with
 $\cX = \euler^{-AT}\colon X \to X$ and $\cY \colon L^2([0,T], U) \to X$, $\cY u =\int_0^T \euler^{- A (T-s)} B(s) u(s) \drm s$.
 Note that apriori $B(s) u(s) \in X_\beta$.
 However, similarly to
 Remark~\ref{lem:w_in_X}
 we see that $\euler^{- A (T-s)} B(s) u(s) \in X$ for almost all $s \in [0,T]$.
 This shows that $\cY$ indeed maps into $X$ and not just into $X_\beta$.
 Boundedness of $\cY$ follows, arguing as in \eqref{eq:spectral_calculus},  from
 \begin{align*}
  \lVert \cY u \rVert^2
  &=
  \int_0^T \lVert \euler^{- A (T - s)} B(s) u(s) \rVert_{X_\beta}^2 \drm s
  \leq
  \int_0^T
  \lVert B(s) u(s) \rVert^2 \drm s
  \\
  &\leq
  \int_0^T
  \lVert B(s) \rVert_{\cL(U, X_\beta)}^2 \rVert u(s) \rVert_U^2 \drm s
  \leq
  \esssup_{t \in [0,T]}
  \lVert B(t) \rVert_{\cL(U, X_\beta)}^2
  \lVert u \rVert_{L^2([0,T], U)}^2 < \infty . \qedhere
 \end{align*}
 \end{proof}
So far, we have only treated the case of non-negative $A$.
The next theorem is an equivalent formulation of Theorem~\ref{thm:obs_and_control} and also treats the situation where $A$ is not assumed to be non-negative any more but merely lower semibounded.
Recall that $\min \sigma(A) = \infspec$.
\begin{theorem}
  \label{cor:upper_bounds_control_cost}
Assume that there are $d_0 > 0$, $d_1 \geq 0$ and $\gamma \in (0,1)$ such that for all $\lambda > \infspec$ and all $\phi \in X$ we have
 \begin{equation*}
  \lVert P_A (\lambda) \phi \rVert^2
  \leq
  d_0 \euler^{d_1 (\lambda - \infspec)^\gamma}
  \lVert B^\ast P_A (\lambda) \phi \rVert_U^2.
 \end{equation*}
Then for all $T > 0$ and all $\phi \in X$ we have the observability estimate
\begin{equation}
 \label{eq:obs_lambda_0}
 \bigl\lVert \euler^{-AT} \phi \bigr\rVert^2
 \leq
 \const^2
 \int_0^T \euler^{- 2 \infspec ( T - t)}  \bigl\lVert B^* \euler^{-At} \phi \bigr\rVert_U^2 \drm t ,
\end{equation}
where $\const$ is as in Theorem~\ref{thm:obs_and_control}.
Moreover, for all $T > 0$, the system~\eqref{eq:abstract_control_system} is null-controllable in time $T$.
Let $K_1 = 2 d_0 \euler^{-\beta} \lVert B \rVert_{\cL (U, X_\beta)}^2 + 1$.
\begin{enumerate}[(a)]
 \item
 If $\infspec < 0$, then the cost satisfies
  \[
  \const^2 %controllable in time\cost^2
  \leq
  \inf_{t \in (0, T]}
   \frac{C_1 d_0}{t}
   K_1^{C_2}
   \exp \left(  C_3
    \left( \frac{d_1 + (- \beta)^{C_4} }{t^\gamma} \right)^{\frac{1}{1 - \gamma}} - 2 \infspec t \right).
 \]
 \item
 If $\infspec = 0$, then the cost satisfies
 \[
  \const^2 %  \cost^2
  \leq
   \frac{C_1 d_0}{T}
   K_1^{C_2}
   \exp \left(  C_3
    \left( \frac{d_1 + (- \beta)^{C_4}}{T^\gamma} \right)^{\frac{1}{1 - \gamma}} \right).
 \]
 \item
 If $\infspec > 0$, then the cost satisfies
 \[
  \const^2 %  \cost^2
  \leq
  \inf_{t \in [0,T)}
   \frac{C_1 d_0}{T - t}
   K_1^{C_2}
   \exp \left(  C_3
    \left( \frac{d_1 + (- \beta)^{C_4} }{(T - t)^\gamma} \right)^{\frac{1}{1 - \gamma}} - 2 \infspec t \right)
   .
 \]
\end{enumerate}
\end{theorem}
The proof of the theorem is deferred to the next section.

In order to investigate the sharpness of the estimates obtained above, we compare them to lower bounds. While these lower bounds are not too difficult to obtain, we provide a proof as a convenience for the reader.
  \begin{theorem}
  \label{thm:lower_bound_control_cost}
  Let $T > 0$ and assume that the system~\eqref{eq:abstract_control_system} is null-controllable in time $T$.
  Then
  \begin{equation*}
   \cost^2
   \geq
   \lVert B \rVert_{\mathcal{L}(U, X_\beta)}^{-2}
   (1 + \infspec^2)^\beta
   \cdot
   \begin{cases}
    \frac{1}{T}
    &
    \text{if $\infspec = 0$},\\[1ex]
    \frac{2 \infspec }{\exp ( 2 \infspec T) - 1}
    &
    \text{if $\infspec \neq 0$}.
   \end{cases}
  \end{equation*}
  \end{theorem}
  \begin{remark} \label{rem:lower-sharp}
   If $B$ is (a multiple of) the identity, one immediately sees from the proof that the bound in Theorem~\ref{thm:lower_bound_control_cost} becomes an equality.
    This means that Theorem~\ref{thm:lower_bound_control_cost} is sharp as a universal lower bound.
  \end{remark}
 \begin{corollary}
  \label{cor:lower_bound_cost_control_cost}
 In the situation of Theorem~\ref{thm:lower_bound_control_cost} we have
   \begin{equation*}
   \cost^2
   \geq
   \lVert B \rVert_{\mathcal{L}(U, X_\beta)}^{-2}
   (1 + \infspec^2)^\beta
   \cdot
   \begin{cases}
    \left( \frac{1}{2 T} - \infspec \right)
    &
    \text{if $\infspec < 0$},\\[1ex]
    \frac{1}{T}
    &
    \text{if $\infspec = 0$},\\[1ex]
    \frac{1}{T} \exp (- 2 \infspec T)
    &
    \text{if $\infspec > 0$}.
   \end{cases}
  \end{equation*}
  Furthermore,
  \begin{equation*}
   \inf_{T > 0}
   \cost^2
   \geq
   \lVert B \rVert_{\mathcal{L}(U, X_\beta)}^{-2}
   (1 + \infspec^2)^\beta
   \cdot
   \begin{cases}
    -2  \infspec
    &
    \text{if $\infspec < 0$},\\[1ex]
    0
    &
    \text{if $\infspec \geq 0$}.
   \end{cases}
  \end{equation*}
 \end{corollary}
\begin{proof}[Proof of Theorem~\ref{thm:lower_bound_control_cost}]
  Since the system~\eqref{eq:abstract_control_system} is null-controllable in time $T$, an observability inequality holds and thus we have
  \[
   \forall \phi\in X \setminus \{0\} \colon \quad
   \int_0^T \lVert B^\ast \euler^{- A t} \phi \rVert_U^2 \drm t \not = 0.
  \]
  Hence, by definition we have
\begin{equation} \label{eq:definition_CT}
 \cost^2
 =
 \sup_{\phi \in X \setminus \{0\}} \frac{\lVert \euler^{- A T} \phi \rVert^2}{\int_0^T \lVert B^\ast \euler^{- A t} \phi \rVert_U^2 \drm t} .
\end{equation}
 Let $\epsilon > 0$ and $0 \not = \phi_0 \in P_A (\infspec + \epsilon)$, where $\infspec = \min \sigma(A)$. By spectral calculus we find
 \[
  \lVert \euler^{- A T} \phi_0 \rVert^2
  =
  \int_\kappa^{\kappa+\epsilon} \euler^{-2 \lambda T} \drm \lVert P_A (\lambda) \phi_0 \rVert^2
  \geq
  \euler^{-2 (\kappa + \epsilon) T} \lVert \phi_0 \rVert^2
 \]
 and
 \begin{align*}
  \int_0^T \lVert B^\ast \euler^{- A t} \phi_0 \rVert_U^2 \drm t
  &\leq
  \lVert B \rVert_{\mathcal{L} (U , X_\beta)}^2 \int_0^T \lVert \euler^{-A t} \phi_0 \rVert^2_{X_{-\beta}} \drm t \\
  &=
  \lVert B \rVert_{\mathcal{L} (U , X_\beta)}^2 \int_0^T \left( \int_{\kappa}^{\kappa + \epsilon} (1+\lambda^2)^{-\beta} \euler^{-2\lambda t} \drm \lVert P_A (\lambda) \phi_0 \rVert^2 \right) \drm t \\
  &\leq
  \lVert B \rVert_{\mathcal{L} (U , X_\beta)}^2 (1+(\kappa + \epsilon)^2)^{-\beta} \lVert \phi_0 \rVert^2  \int_0^T \euler^{-2\kappa t} \drm t .
 \end{align*}
For the latter integral we obtain $T$ if $\kappa = 0$, and $(1-\euler^{-2\kappa T}) / (2\kappa)$ if $\kappa \not = 0$. We choose $\phi = \phi_0$ in Eq.~\eqref{eq:definition_CT} and obtain an $\epsilon$-dependent lower bound on $\cost$. The statement of the theorem follows since $\epsilon > 0$ is arbitrary.
\end{proof}
  In particular, we see from Corollary~\ref{cor:lower_bound_cost_control_cost} that if $\infspec < 0$, then $C_\infty := \inf_{t > 0} \cost$ is strictly positive.
  This is in contrast to the situation $\infspec \geq 0$, where $C_\infty = 0$, i.e.\ the control cost vanishes in the large time limit.
\begin{remark}
  \label{rem:table}
  Let us now compare the lower bounds from Theorem~\ref{thm:lower_bound_control_cost} with the upper bounds from Theorems~\ref{thm:obs_and_control} and \ref{cor:upper_bounds_control_cost} in the special case $\beta = 0$.
  We focus on this case since in all our applications below we have $\beta = 0$.
\par
 In Table~\ref{table:asymptotics}, we summarize the asymptotic behavior of the upper and lower bounds on the control cost in the large and small time asymptotic regime.
 We only keep track of the parameters $T$, $d_1$ and $\infspec$ and omit multiplicative constants depending only on $d_0$, $\gamma$, and $\lVert B \rVert_{\cL(U, X)}$.
 The parameter $C$ stands for a constant which only depends on the parameter $\gamma$, and might change from case to case.
\par
In the case $\infspec < 0$ or $\infspec > 0$, the upper bounds in Theorem~\ref{cor:upper_bounds_control_cost} are given in terms of infima over $t \in (0,T]$ or $t \in (0,T)$, respectively.
In order to obtain the upper bounds in the table, for $T \to 0$ we choose $t = T/2$, while in the regime $T \to \infty$ we choose $t =  (-\infspec)^{-1}$ if $\infspec < 0$ and $t = T - 1$ if $\infspec > 0$.
\begin{table}[ht]\centering
\begin{tabular}{cc>{\centering\arraybackslash}p{5cm}>{\centering\arraybackslash}p{5cm}}
\toprule
                                 &              & lower bound  & upper bound \\
\midrule
 \multirow{2}{*}{\vspace{-4ex}$\infspec < 0$} & $T\to\infty$ & $-\infspec$               & $(-\infspec) \exp \left( C d_1^{\frac{1}{1 - \gamma}} (-\infspec)^\frac{\gamma}{1 - \gamma} \right)$ \\[2ex]
                                 & $T\to 0$     & $T^{-1}$    & $T^{-1} \exp \left( C \left( \frac{d_1}{T^{\gamma}} \right)^{\frac{1}{1- \gamma}} \right)$ \\
\midrule
 \multirow{2}{*}{\vspace{-4ex}$\infspec = 0$} & $T\to\infty$ &    $T^{-1}$          &  $T^{-1}$  \\[2ex]
                                 & $T\to 0$     & $T^{-1}$   &  $T^{-1} \exp \left( C \left( \frac{d_1}{T^{\gamma}} \right)^{\frac{1}{1- \gamma}} \right)$ \\
\midrule
 \multirow{2}{*}{\vspace{-4ex}$\infspec > 0$} & $T\to\infty$ & $T^{-1}\euler^{-2\infspec T}$ &   $\euler^{-2\infspec T} \exp \left( C d_1^{\frac{1}{1 - \gamma}} + 2 \infspec \right)$ \\[2ex]
                                 & $T\to 0$     & $T^{-1}$           & $T^{-1} \exp \left( C \left( \frac{d_1}{T^{\gamma}} \right)^{\frac{1}{1- \gamma}} \right)$ \\
\bottomrule
\end{tabular}
\caption{Asymptotic behavior of lower and upper bounds on $\cost^2$ in the case $\beta = 0$ \label{table:asymptotics}}
\end{table}
To discuss these bounds, let us first consider the case $d_1 = 0$.
This implies that we have full control in the sense that the control operator $B$ is boundedly invertible.
In this situation, the upper and lower bounds in Table~\ref{table:asymptotics} coincide except for the case when $\infspec > 0$ in the regime $T \to \infty$.
\par
Let us now assume $d_1 > 0$.
The lower bounds are consistent with Theorem~\ref{thm:lower_bound_control_cost} and cannot be improved in general, as Remark~\ref{rem:lower-sharp} shows.
In the large time regime, the upper and lower bounds exhibit qualitatively the same asymptotic behavior except for the case when $\infspec > 0$.
The different asymptotic behavior which we observe in the small time regime cannot be avoided.
In fact, there exist examples where the exponential blowup of the type $\exp (C T^{-\gamma/(1 - \gamma)})$ indeed occurs, see e.g.\ \cite{FernandezZ-00,Miller-04}. They consider the controlled heat equation with control in a subset of the domain, see Section~\ref{sec:applications} for details on the controlled heat equation. Note that this example corresponds to $\gamma = 1/2$.
This shows in particular that the upper bounds in Table~\ref{table:asymptotics} are sharp in this regime.
\end{remark}
\begin{remark} \label{remark:full_control}
 Let $X = U$, $\beta = 0$ and $B = I$.
 In this case one can explicitly construct null-control functions in time $T > 0$. We give two examples.
 The first one is given by
 \begin{equation*}
  u_1 (t)
  =
  \int_{\infspec}^\infty f_T(\lambda) \drm P_A(\lambda) w_0,
  \quad
  \text{where}
  \quad
  f_T(\lambda)
  =
  \begin{cases}
   - T^{-1}\quad &\text{if $\lambda = 0$},\\
   \frac{- \lambda}{\euler^{\lambda T} - 1} &\text{if $\lambda \neq 0$}.
   \end{cases}
  \end{equation*}
  The second one is given by
   \begin{equation*}
  u_2 (t)
  =
  \int_{\infspec}^\infty \euler^{\lambda t} g_T(\lambda) \drm P_A(\lambda) w_0,
  \quad
  \text{where}
  \quad
  g_T(\lambda)
  =
  \begin{cases}
   - T^{-1}\quad &\text{if $\lambda = 0$},\\
   \frac{- 2 \lambda}{\euler^{2 \lambda T} - 1} &\text{if $\lambda \neq 0$}.
   \end{cases}
  \end{equation*}
 The fact that $u_1$ and $u_2$ are null-control functions in time $T$ follows from the Duhamel formula~\eqref{eq:Duhamel_formula} and spectral calculus.
 Note that $u_1$ is time-independent while $u_2$ is time-dependent.
 Moreover, it follows that
 $\cost \leq \lVert u_i \rVert_{L^2([0,T],X)}$, $i \in \{1 , 2 \}$.
 We estimate
 \begin{equation}
  \label{eq:Cobs_upper_bound_explicit_feedback_u_1}
 \cost^2
 \leq
 \lVert u_1 \rVert_{L^2([0,T], U)}^2
 \leq
  \begin{cases}
    T^{-1} & \text{if $\infspec = 0$},\\
    T   \big\lvert \frac{\infspec}{\euler^{\infspec T} - 1} \big\rvert^2 & \text{if $\infspec \neq 0$}
  \end{cases}
 \end{equation}
 and
 \begin{equation}
  \label{eq:Cobs_upper_bound_explicit_feedback_u_2}
  \cost^2
  \leq
  \lVert u_2 \rVert_{L^2([0,T],X)}^2
  \leq
    \begin{cases}
   T^{-1}\quad &\text{if $\infspec = 0$},\\
   \frac{2 \infspec}{\euler^{2 \infspec T} - 1} &\text{if $\infspec \neq 0$}.
   \end{cases}
 \end{equation}
 Since the upper bound in~\eqref{eq:Cobs_upper_bound_explicit_feedback_u_2} coincides with the lower bound in Theorem~\ref{thm:lower_bound_control_cost}, we conclude that $u_2$ is the (unique) null-control function in time $T$ with minimal norm.
 Furthermore, the inequalities in~\eqref{eq:Cobs_upper_bound_explicit_feedback_u_2} are actually equalities.
\par
 If $\infspec = 0$ the bounds in \eqref{eq:Cobs_upper_bound_explicit_feedback_u_1} and~\eqref{eq:Cobs_upper_bound_explicit_feedback_u_2} coincide.
 Hence, in this case, the optimal null-control function in time $T$ is a time-independent function.
\par
 We also see that for certain choices of $T$ and $\infspec$, there is a constant-in-time null-control function in time $T$ with norm which is close to the optimal one.
 This is related to the so-called turnpike property, see, for example \cite{TrelatZZ18}.
\end{remark}
\section{Proofs of Theorems~\ref{thm:obs_and_control} and  \ref{cor:upper_bounds_control_cost}}
\label{sec:proofs}

\begin{proof}[Proof of Theorem~\ref{thm:obs_and_control}]
Let $T > 0$.
For $\phi \in X\subset X_\beta$, $t \in (0,T]$, and $\lambda > 0$ we use the notation
 \begin{align*}
   F (t) &= \bigl\lVert \euler^{-At} \phi \bigr\rVert^2 ,
 & F_\lambda (t) &= \bigl\lVert \euler^{-At} P_A (\lambda) \phi \bigr\rVert^2 ,
 & F_\lambda^\perp (t) &= \bigl\lVert \euler^{-At} (I - P_A (\lambda)) \phi \bigr\rVert^2 , \\
   G (t) &= \bigl\lVert B^* \euler^{-At} \phi \bigr\rVert_U^2 ,
 & G_\lambda (t) &= \bigl\lVert B^* \euler^{-At} P_A (\lambda) \phi \bigr\rVert_U^2 ,
 & G_\lambda^\perp (t) &=\bigl\lVert B^* \euler^{-At} (I - P_A (\lambda)) \phi \bigr\rVert_U^2 .
 \end{align*}
Since $A \geq 0$ we have $F (t_1) \geq F (t_2)$, $F_\lambda (t_1) \geq F_\lambda (t_2)$, and $F_\lambda^\perp (t_1) \geq F_\lambda^\perp (t_2)$ if $t_1 \leq t_2$ and $\lambda > 0$. By monotonicity and our assumption \eqref{eq:UCP}, we obtain for all $t \in (0,T]$ and all $\lambda > 0$
\begin{equation} \label{eq:Flambda}
 F_\lambda (t)
 =    \frac{2}{t} \int_{t/2}^t F_\lambda (t) \drm \tau
 \leq \frac{2}{t} \int_{t/2}^t F_\lambda (\tau) \drm \tau
 \leq \frac{2 d_0 \euler^{d_1 \lambda^\gamma}}{t}  \int_{t/2}^t G_\lambda (\tau) \drm \tau .
\end{equation}
By spectral calculus and since $\lVert B^* \rVert_{\mathcal{L} (X_{-\beta} , U)} = \lVert B \rVert_{\mathcal{L} (U , X_\beta)}$ we have
 \begin{align}\label{eq:spectral_calculus}
  G_\lambda^\perp (t)
  &\leq \lVert B \rVert_{\cL (U , X_{\beta})}^2 \lVert \euler^{-At} (I-P_A (\lambda)) \phi \rVert_{X_{-\beta}}^2 \nonumber \\
  &= \lVert B \rVert_{\cL (U , X_{\beta})}^2 \lVert (I + A^2)^{-\beta / 2} \euler^{-At} (I-P_A(\lambda)) \phi \rVert^2 \nonumber \\
  & =  \lVert B \rVert_{\cL (U , X_{\beta})}^2  \int_\lambda^\infty (1+\mu^2)^{-\beta} \euler^{- 2\mu t} \drm \lVert P_A (\mu) \phi \rVert^2 .
  \end{align}
Note that this justifies that
$\euler^{-At} (I-P_A (\lambda)) \phi$ is indeed in $X_{-\beta}$.
Recall that $\beta < 0$. Let $\Theta > 0$ to be specified later. For $\mu,t > 0$ we estimate
\[
(1+\mu^2)^{-\beta} \euler^{- \mu t}  \leq \left(1+ \left(-\frac{2\beta}{t} \right)^2 \right)^{-\beta} \leq \exp \left( \frac{C_\Theta}{t^{\Theta}} - \beta  \right), \quad C_\Theta = 2^{\Theta} (-\beta)^{\Theta + 1}  \left(\frac{2+\Theta}{\Theta} \right),
\]
where the first inequality follows by maximizing with respect to $\mu$, and the second one follows from the inequality $\ln (1+x) \leq (2/\Theta + 1) x^{\Theta/2} + 1$ for $x \geq 0$.
Hence,
  \begin{align} \label{eq:GlambdaPerp}
   G_\lambda^\perp (t)& \leq \lVert B \rVert_{\cL (U , X_{\beta})}^2 \int_\lambda^\infty  \euler^{C_\Theta / t^\Theta-\beta- \mu t} \drm \lVert P_A (\mu) \phi \rVert^2
  \leq  \lVert B \rVert_{\cL (U , X_{\beta})}^2 \euler^{C_\Theta / t^\Theta - \beta -\lambda t/2} F (t / 2) .
 \end{align}
 Similarly we find
 \[
 F_\lambda^\perp (t) = \int_\lambda^\infty \euler^{-2 \mu t}  \drm \lVert P_A (\mu) \phi \rVert^2
  \leq \euler^{- 3 \lambda t/2}\int_\lambda^\infty \euler^{-\mu t/2}  \drm \lVert P_A (\mu) \phi \rVert^2
 \leq \euler^{- 3\lambda t/2} F (t/4) .
 \]
>From the last inequality and Ineq.~\eqref{eq:Flambda} we obtain
\[
 F (t) = F_\lambda (t) + F_\lambda^\perp (t) \leq  \frac{2 d_0 \euler^{d_1 \lambda^\gamma}}{t}  \int_{t/2}^t G_\lambda (\tau) \drm \tau
 +  \euler^{-3\lambda t / 2} F (t / 4) .
\]
Since $G_\lambda (t) \leq 2(G_\lambda^\perp (t) + G (t))$ and by Ineq.~\eqref{eq:GlambdaPerp}
we obtain for all $t \in (0,T]$ and all $\lambda > 0$
\begin{align*}
 F (t) &\leq \frac{4 d_0 \euler^{d_1 \lambda^\gamma}}{t}  \int_{t/2}^t (G_\lambda^\perp (\tau) + G (\tau)) \drm \tau
 +  \euler^{-3\lambda t / 2} F (t / 4) \\
&\leq \frac{4 d_0 \euler^{d_1 \lambda^\gamma}}{t}  \int_{t/2}^t  G (\tau) \drm \tau
+\frac{4 d_0 \euler^{-\beta} \euler^{d_1 \lambda^\gamma} \lVert B \rVert_{\cL (U,X_\beta)}^2}{t}    \int_{t/2}^t \frac{F (\tau / 2)}{\euler^{\lambda \tau / 2 - C_\Theta / t^\Theta}}   \drm \tau
 +  \frac{F (t /4) }{\euler^{3\lambda t / 2}} .
 \end{align*}
Since $F (\tau / 2) \leq F (t / 4)$, $\euler^{-\lambda \tau / 2} \leq \euler^{- \lambda t / 4}$, and $\euler^{C_\Theta / \tau^{\Theta}} \leq \euler^{2^\Theta C_\Theta / t^\Theta}$ for $\tau \geq t/2$, we obtain
 \begin{align*}
 F (t)
 &\leq
 \frac{4 d_0 \euler^{d_1 \lambda^\gamma}}{t}
 \int_{t/2}^t  G (\tau) \drm \tau
  +\euler^{-\lambda t / 4 + 2^\Theta C_\Theta /t^\Theta}
  \left( 2  d_0 \euler^{-\beta} \euler^{d_1 \lambda^\gamma} \lVert B \rVert_{\cL (U,X_\beta)}^2
    + 1
  \right)
 F (t / 4) \\
 &\leq
  \frac{4 d_0 \euler^{d_1 \lambda^\gamma}}{t}
  \int_{t/2}^t  G (\tau) \drm \tau
  +
  \euler^{- \lambda t / 4 + 2^\Theta C_\Theta / t^\Theta +d_1 \lambda^\gamma}
  \left(
    2 d_0 \euler^{-\beta} \lVert B \rVert_{\cL (U,X_\beta)}^2
    +
    1
 \right) F (t / 4) .
\end{align*}
With the notation
\begin{align*}
 D_1 (t , \lambda)
 &=
 \frac{4 d_0 \euler^{d_1 \lambda^\gamma}}{t}  \int_{t/2}^t  G (\tau) \drm \tau , \\
 \intertext{and}
 D_2 (t , \lambda)
 &=
 \euler^{- \lambda t / 4 + 2^\Theta C_\Theta / t^\Theta +d_1 \lambda^\gamma}
  \left(
    2 d_0 \euler^{-\beta} \lVert B \rVert_{\cL (U,X_\beta)}^2
    +
    1
 \right)
\end{align*}
we can summarize that for all $t \in (0,T]$ we have
\begin{equation} \label{eq:before_iterating}
F (t) \leq D_1 (t , \lambda) + D_2 (t , \lambda) F (t / 4).
\end{equation}
This inequality can be iterated. For $k \in \NN_0$ let $\lambda_k = \nu \alpha^k$ with $\nu > 0$ and $\alpha > 1$ to be specified later.
In particular, applying Ineq.~\eqref{eq:before_iterating} with $t = T$ and $\lambda = \lambda_0$ at the first place, the term $F (4^{-1}T)$ on the right hand side can then be estimated by Ineq.~\eqref{eq:before_iterating} with $t = 4^{-1} T$ and $\lambda = \lambda_1$. This way, we obtain after two steps
\begin{align*}
 F (T) &\leq  D_1 (T , \lambda_0) + D_2 (T , \lambda_0) \left(  D_1 (4^{-1} T , \lambda_1) + D_2 (4^{-1}T , \lambda_1) F (4^{-2}T) \right) \\
 & =  D_1 (T , \lambda_0) + D_1 (4^{-1}T , \lambda_1) D_2 (T , \lambda_0) + D_2 (T , \lambda_0) D_2 (4^{-1}T , \lambda_1) F (4^{-2}T ) .
\end{align*}
After $N + 1$ steps of this type we obtain
\begin{equation} \label{eq:after_iteration}
 F (T) \leq D_1 (T , \lambda_0) + \sum_{k=1}^N D_1 (4^{-k}T , \lambda_k) \prod_{l = 0}^{k-1} D_2 (4^{-l}T , \lambda_l)
 + F (4^{-N-1} T) \prod_{k=0}^N D_2 (4^{-k}T, \lambda_k) .
\end{equation}
In order to study the limit $N \to \infty$, we assume that $4^{\Theta+1} \leq \alpha$, $\alpha^\gamma \leq \alpha / 4$, and
$
 \nu T
 >
 2^{\Theta+2} C_\Theta T^{-\Theta}
 +
 d_1 \nu^\gamma \alpha.
$
This ensures that the constants
\begin{equation}\label{eq:Ki}
K_1 = 2 d_0 \euler^{-\beta} \lVert B \rVert_{\cL (U,X_\beta)}^2 + 1 ,
\quad
K_2 =  \nu  T / 4 - 2^\Theta C_\Theta / T^\Theta - d_1 \nu^\gamma
,
\quad
K_3 = \frac{K_2}{\alpha / 4 - 1} - d_1 \nu^\gamma
\end{equation}
are positive.
Then we have that
\begin{align}
\label{eq:last_term}
 \prod_{k=0}^N D_2 (4^{-k} T , \lambda_k)
 &=
 K_1^{N+1} \prod_{k=0}^N  \euler^{- \nu (\alpha / 4)^k T / 4 + 2^\Theta C_\Theta 4^{\Theta k} / T^\Theta +d_1 \nu^\gamma(\alpha^\gamma)^k}  \nonumber \\
 &\leq
 K_1^{N+1} \prod_{k=0}^N  \euler^{(\alpha / 4)^k ( - \nu  T / 4 + 2^\Theta C_\Theta / T^\Theta +d_1 \nu^\gamma)}
 =
K_1^{N+1} \prod_{k=0}^N  \euler^{- K_2 (\alpha / 4)^k }.
\end{align}
Since $K_1 , K_2 > 0$ and $\alpha > 4$ this tends to zero as $N$ tends to infinity. From Ineq.~\eqref{eq:last_term} and the definitions of $D_1 (4^{-k} T , \lambda_k)$ and $K_3$, we infer that the middle term of the right hand side of Ineq.~\eqref{eq:after_iteration} obeys the upper bound
% \begin{align} \label{eq:middle_term}
%  \sum_{k=1}^N   D_1 &(4^{-k} T , \lambda_k)  \prod_{l = 0}^{k-1} D_2 (4^{-l} T , \lambda_l) \nonumber \\
%   &\leq \int_{0}^T  G (\tau) \drm \tau \sum_{k=1}^N \frac{4^{k+1} d_0 \exp(d_1 \nu^\gamma (\alpha / 4)^{k})}{T}   K_1^{k}
%  \exp \left( - K_2 \frac{(\alpha / 4)^k - 1}{\alpha / 4 - 1}   \right) \nonumber \\
%   & = \int_{0}^T  G (\tau) \drm \tau \frac{4 d_0}{T} \exp\left( \frac{K_2}{\alpha / 4 - 1} \right) \sum_{k=1}^N \left(4 K_1 \right)^{k}
%   \exp\left( - K_3 (\alpha / 4)^k \right) .
% \end{align}
\begin{align} \label{eq:middle_term}
 \sum_{k=1}^N   & D_1 (4^{-k} T , \lambda_k)  \prod_{l = 0}^{k-1} D_2 (4^{-l} T , \lambda_l)
  \\ &\leq \int_{0}^T  G (\tau) \drm \tau \sum_{k=1}^N \frac{4^{k+1} d_0 \exp(d_1 \nu^\gamma (\alpha / 4)^{k})}{T}   K_1^{k}
 \exp \left( - K_2 \frac{(\alpha / 4)^k - 1}{\alpha / 4 - 1}   \right) \nonumber \\
  & = \int_{0}^T  G (\tau) \drm \tau \frac{4 d_0}{T} \exp\left( \frac{K_2}{\alpha / 4 - 1} \right) \sum_{k=1}^N \left(4 K_1 \right)^{k}
  \exp\left( - K_3 (\alpha / 4)^k \right) .
\end{align}
Letting $N$ tend to infinity we obtain from Ineqs.~\eqref{eq:after_iteration}, \eqref{eq:last_term} and \eqref{eq:middle_term} that
\begin{equation*}
 \bigl\lVert \euler^{-AT} \phi \bigr\rVert^2
 \leq
 \tilde C_{\mathrm{obs}}^2
 \int_0^T    \bigl\lVert B^* \euler^{-At} \phi \bigr\rVert_U^2 \drm t ,
\end{equation*}
where
\begin{equation} \label{eq:Cobs}
 \tilde C_{\mathrm{obs}}^2 =
\frac{4 d_0 \euler^{d_1 \nu^\gamma}}{T} +
\frac{4 d_0}{T} \exp\left( \frac{K_2}{\alpha / 4 - 1} \right) \sum_{k=1}^\infty \left( 4K_1 \right)^{k}  \exp\left( - K_3 (\alpha / 4)^k \right) .
\end{equation}
We choose $\Theta$, $\alpha$ and $\nu$ as in \eqref{eq:parameters} and conclude the observability inequality~\eqref{eq:obs} from Lemma~\ref{lem:Cobs}.
\par
Since \eqref{eq:obs} corresponds to part (b) of Lemma~\ref{lem:absop} with
 $\cX = \euler^{-AT}\colon X \to X$ and $\cY = \cB^T \colon L^2([0,T], U) \to X$.
the system is null-controllable in time $T$. By the definition of $\cost$ we have $\cost \leq C_\obs$.
\end{proof}

\begin{lemma}\label{lem:Cobs}
Let $d_0 > 0$ , $d_1 \geq 0$, $\gamma \in (0,1)$, $T > 0$,
\begin{equation}\label{eq:parameters}
  \Theta
  =
  \frac{\gamma^2}{1 - \gamma},
  \quad
  \alpha
  =
  8 \cdot 4^{\frac{1}{1 - \gamma}},
  \quad\text{and}\quad
  \nu
  =
  \left( \frac{\alpha d_1}{T} + \frac{D}{T^{1- \gamma}} + \frac{E}{T} \right)^{\frac{1}{1 - \gamma}},
 \end{equation}
 where
 \begin{equation*}
  D
  =
   (3 \alpha \ln (4K_1))^{1- \gamma}
  , \quad
  E
  =
  \left(
  \frac{8 \cdot 2^\Theta C_\Theta}{D}
  \right)^{\frac{1 - \gamma}{\gamma}}
  ,
  \quad
  C_\Theta = 2^\Theta (-\beta)^{\Theta + 1} \left(\frac{2+\Theta}{\Theta} \right) ,
\end{equation*}
 and $K_1 = 2 d_0 \euler^{-\beta} \lVert B \rVert_{\cL (U, X_\beta)}^2 + 1$.
Then we have $4^{\Theta+1} \leq \alpha$, $\alpha^\gamma \leq \alpha / 4$, and
 $ \nu T
 >
 2^{\Theta+2} C_\Theta T^{-\Theta}
 +
 d_1 \nu^\gamma \alpha$.
Moreover, for all $T>0$ the constant $\tilde C_{\mathrm{obs}}^2$ from \eqref{eq:Cobs} satisfies
\[
  \tilde C_{\mathrm{obs}}^2
  \leq
  \frac{C_1 d_0}{T}
  K_1^{C_2}
  \exp \left(  C_3
   \left( \frac{d_1 + (-\beta)^{C_4}}{T^\gamma} \right)^{\frac{1}{1 - \gamma}} \right) .
\]
Here, $C_i > 0$, $i \in \{1,2,3,4\}$, are constants depending only on $\gamma$.
They are explicitly given by Eq.~\eqref{eq:final_constant_lemma}.
\end{lemma}
\begin{proof}
It is easy to see that $4^{\Theta+1} \leq \alpha$, and $\alpha^\gamma \leq \alpha / 4$.
For the constant $K_3$ from \eqref{eq:Ki} we have
 \begin{align*}
  K_3
  &=
  \frac{\nu T/4 - 2^\Theta C_\Theta/T^\Theta - d_1 \nu^\gamma \alpha / 4}{ ( \alpha / 4 - 1)} \\
   &=
   \frac{\nu^\gamma}{\alpha - 4}
   \left[
   	\left( \frac{\alpha d_1}{T} + \frac{D}{T^{1- \gamma}} + \frac{E}{T} \right) T
   	-
   	\frac{4 \cdot 2^\Theta C_\Theta}{T^\Theta}
   	\left( \frac{\alpha d_1}{T} + \frac{D}{T^{1- \gamma}} + \frac{E}{T} \right)^{-\frac{\gamma}{1 - \gamma}}
   	- d_1 \alpha
   \right]\\
   &
   \geq
   \frac{\nu^\gamma}{\alpha - 4}
   \left[
   	D T^\gamma
   	-
   	\frac{4 \cdot 2^\Theta C_\Theta}{T^\Theta}
   	\frac{E^{\frac{-\gamma}{1 - \gamma}}}{T^{\frac{-\gamma}{1 - \gamma}}}
   \right]
   = \frac{\nu^\gamma D T^\gamma}{2(\alpha - 4)}.
 \end{align*}
This shows in particular that $ \nu T > 2^{\Theta+2} C_\Theta T^{-\Theta} + d_1 \nu^\gamma \alpha$.
We further estimate
\[
	K_3
	\geq
	\frac{\left( \frac{D}{T^{1-\gamma}} \right)^{\gamma / (1-\gamma)} D T^\gamma}{2(\alpha - 4)}
	=
	\frac{D^{1/(1-\gamma)}}{2 (\alpha - 4)} .
\]
For the constant $K_2$ from \eqref{eq:Ki} we estimate using $\alpha \geq 8$
 \begin{align*}
  \frac{K_2}{\alpha / 4 - 1}
  &
  \leq \nu T / 2
  =
  \frac{T}{2}
  \left(
  	\frac{\alpha d_1}{T} + \frac{D}{T^{1- \gamma}} + \frac{E}{T}
  \right)^{\frac{1}{1 - \gamma}}
  \leq
  \frac{\alpha^{\frac{1}{1 - \gamma}}}{2}
  \left(
  	\frac{\alpha d_1  + E }{T^\gamma} + D
  \right)^{\frac{1}{1 - \gamma}} .
 \end{align*}
Let us now note that for all $A > 1$, and $B > 0$ we have
 \begin{equation}
 \label{eq:estimating_ABC}
  \sum_{k = 1}^\infty A^k \euler^{- B 2^k}
  \leq
  \left( \frac{2 \ln A}{B \euler \ln 2 } \right)^{\frac{\ln A}{\ln 2}}
  \frac{1}{B} ,
\end{equation}
since
\[
  \sum_{k = 1}^\infty \euler^{- \frac{B}{2} 2^k}
 \leq
  \sum_{k = 1}^\infty \euler^{- k B}
 =
  \frac{\euler^{- B}}{1 - \euler^{- B}}
 =
  \frac{1}{\euler^{B} - 1}
 \leq
  \frac{1}{B}
 \]
 and
\[
  \sum_{k = 1}^\infty A^k \euler^{- B 2^k}
 \leq
  \sup_{x \geq 1} ( A^x \euler^{- \frac{B}{2} 2^x} )
  \sum_{k = 1}^\infty \euler^{- \frac{B}{2} 2^k}
 =
  \left( \frac{2 \ln A}{B \euler \ln 2} \right)^{\frac{\ln A}{\ln 2}}
  \sum_{k = 1}^\infty \euler^{- \frac{B}{2} 2^k} .
\]
We use $\alpha \geq 8$ and apply Ineq.~\eqref{eq:estimating_ABC} with $A = 4K_1 > 1$, and $B = K_3$ to obtain
  \begin{equation}
   \label{eq:application_estimating_ABC_1}
   \sum_{k=1}^\infty \left(4 K_1\right)^{k}  \exp\left( - K_3 (\alpha / 4)^k \right)
   \leq
   \sum_{k=1}^\infty \left(4 K_1\right)^{k}  \exp\left( - K_3 2^k \right)
   \leq
   \left( \frac{2 \ln (4 K_1)}{ K_3 \euler \ln 2} \right)^{\frac{\ln (4 K_1)}{\ln 2}} \frac{1}{K_3}.
  \end{equation}
  By the above estimate on $K_3$ and since $\alpha > 4$ we find
  \[
   \frac{2 \ln (4K_1 )}{ K_3 \euler \ln 2}
   \leq
   \frac{2}{\euler \ln 2}
     \frac{\ln (4 K_1) 2 (\alpha - 4 )}{D^{1/(1-\gamma)}}
     = \frac{2}{\euler \ln 2}
     \frac{ 2 (\alpha - 4 )}{3 \alpha}
    \leq 1.
  \]
  Note that the exponent $\ln (4 K_1) / \ln 2$ in~\eqref{eq:application_estimating_ABC_1} is positive, and that $D \geq 1$.
  Hence, the right hand side of~\eqref{eq:application_estimating_ABC_1} is bounded from above by $2 (\alpha - 4)$.
  Using this, $\alpha \geq 8$, and $d_1 \nu^\gamma \leq d_1 \nu^\gamma + K_3 = K_2 / (\alpha / 4 - 1)$, we find
  \begin{align*}
   \tilde C_{\mathrm{obs}}^2
   & =
   \frac{4 d_0 \euler^{d_1 \nu^\gamma}}{T} +
\frac{4 d_0}{T} \exp\left( \frac{K_2}{\alpha / 4 - 1} \right) \sum_{k=1}^\infty \left(4 K_1\right)^{k}  \exp\left( - K_3 (\alpha / 4)^k \right) \\
&\leq
   \frac{4 d_0}{T}
   ( 1 + K_3^{-1})
   \exp \left( \frac{K_2}{\alpha / 4 - 1} \right) \\
   &\leq
   \frac{4 d_0}{T} \left( 1 + 2 (\alpha -4) \right)
   \exp \left(  \frac{\alpha^{\frac{1}{1 - \gamma}}}{2}
  \left(
  	\frac{\alpha d_1  + E }{T^\gamma} + D
  \right)^{\frac{1}{1 - \gamma}} \right)  .
  \end{align*}
  Since $(a+b)^x \leq 2^{x-1} (a^x + b^x)$ for $x > 1$ and $a,b \geq 0$ we obtain
  \begin{multline}
   \label{eq:final_constant_lemma}
   \tilde C_{\mathrm{obs}}^2
   \leq
   \frac{4 d_0}{T}
   \left( 1 + 2 (\alpha - 4) \right)
   \left( 4 K_1 \right)^{3 \alpha^{\frac{2 - \gamma}{1 - \gamma}} 2^{2 \Theta + 3}}
 	\\
 	\times
	\exp \left(
		\alpha^{\frac{2}{1 - \gamma}} 4^{\frac{\gamma + \Theta + 2}{1 - \gamma}}
		\left( \frac{\Theta + 2}{\Theta} \right)^{\frac{1}{1 - \gamma}}
		\left( \frac{d_1 + (-\beta)^{\Theta + 1}}{T^\gamma} \right)^{\frac{1}{1 - \gamma}}
		\right).
  \end{multline}
	%\qedhere
\end{proof}

\begin{proof}[Proof of Theorem~\ref{cor:upper_bounds_control_cost}]
 Since $P_A(\lambda) = P_{A - \infspec}(\lambda - \infspec)$, we have by assumption for all $\lambda \geq 0$ that
 \[
  \lVert P_{A - \infspec} (\lambda) \phi \rVert^2
  \leq
  d_0 \euler^{d_1 \mu^\gamma}
  \lVert B^\ast P_{A - \infspec} (\lambda) \phi \rVert_U^2.
 \]
 Since the operator $A - \infspec$ is non-negative, we obtain from Theorem~\ref{thm:obs_and_control} the observability estimate
 \[
  \bigl\lVert \euler^{-(A - \infspec) T} \phi \bigr\rVert^2
  \leq
  \const^2
  \int_0^T  \bigl\lVert B^* \euler^{- (A - \infspec) t} \phi \bigr\rVert_U^2 \drm t
  =
  \const^2
  \int_0^T \euler^{2 \infspec t}  \bigl\lVert B^* \euler^{- A t} \phi \bigr\rVert_U^2 \drm t,
 \]
 where $\const$ is as in Theorem~\ref{thm:obs_and_control}.
 Dividing by $\euler^{2 \infspec T}$ yields~\eqref{eq:obs_lambda_0}.
\par
 If $\infspec < 0$, we have for all $t > 0$
 \[
 \bigl\lVert \euler^{-At} \phi \bigr\rVert^2
 \leq
 \const^2
 \euler^{- 2 \infspec t}
 \int_0^t \bigl\lVert B^* \euler^{-A\tau} \phi \bigr\rVert_U^2 \drm \tau.
 \]
 Using the equivalence between observability and null-controllability as in the proof of Theorem~\ref{thm:obs_and_control}, we conclude that the system~\eqref{eq:abstract_control_system} is null-controllable in time $t$ for all $t > 0$ with cost satisfying
 \[
  C_{t}^2
  \leq
  \frac{C_1 d_0}{t}
   K_1^{C_2}
   \exp \left(  C_3
    \left( \frac{d_1 + (- \beta)^{C_4} }{t^\gamma} \right)^{\frac{1}{1 - \gamma}} - 2 \infspec t \right) .
 \]
 Note that this expression grows exponentially as $t$ tends to infinity.
 However, if the system is null-controllable in time $t$ with cost $C_t$, then it is also null-controllable in time $T > t$ with the same cost $\cost = C_t$.
 For any $t \in (0,T]$, we can choose a null-control function in time $t$, and apply no control in $(t, T]$.
 This yields the upper bound in this case.
\par
 The case $\infspec = 0$ is the statement of Theorem~\ref{thm:obs_and_control}.
 If $\infspec > 0$, then we choose any $t \in (0, T)$, apply no control in $[0, t]$ and find
 \[
  \lVert w(t) \rVert
  \leq
  \euler^{- \infspec t}
  \lVert w_0 \rVert.
 \]
 Then, we apply Theorem~\ref{thm:obs_and_control} with initial state $w(t)$ in the time interval $[t,T]$.
\end{proof}
\section{Spectral inequalities and explicit cost for the controlled heat equation}
\label{sec:applications}
In this section, we apply the results from Section~\ref{sec:abstract} to the controlled heat equation with heat generation term on bounded and unbounded domains.
More precisely, our setting is as follows.
\par
Let $d \in \NN$, $\alpha_i ,\beta_i \in \RR\cup \{\pm \infty\}$ with $\beta_i - \alpha_i > 0$, and
\begin{equation}\label{eq:Omega}
 \Omega = \bigtimes_{i =1}^d (\alpha_i , \beta_i) .
\end{equation}
We denote by $-\Delta$ the self-adjoint Laplace operator in $L^2 (\Omega)$ with Dirichlet, Neumann or periodic boundary conditions. Here we allow for periodic boundary conditions only if $\Omega = \Lambda_L = (-L/2,L/2)^d$ for some $L > 0$. Moreover, let $V \in L^\infty(\Omega)$ be real-valued, and
define the self-adjoint operator $H_\Omega$ in $L^2 (\Omega)$ by
\begin{equation*}
 H_\Omega = -\Delta + V .
\end{equation*}
In $L^2 (\Omega)$ we consider the controlled heat equation with heat generation term $(-V)$
\begin{equation} \label{eq:heat_equation}
\dot w + H_\Omega w = \Eins_{S\cap \Omega} u ,
\quad
w(0,\cdot) = w_0 \in L^2(\Omega),
 \end{equation}
where $T > 0$, $w,u \in L^2 ([0,T] \times \Omega)$, and where $S$ is non-empty and measurable, usually given by a $(\rho,a)$-thick set or a $(G , \delta)$-equidistributed set, see below for definitions.
Note that we simultaneously treat bounded and unbounded domains such as $\RR^d$, half-spaces, infinite strips, or hypercubes.
\par
Theorems~\ref{thm:obs_and_control} and \ref{cor:upper_bounds_control_cost} translate spectral inequalities into null-controllability of the corresponding controlled Cauchy problem with explicit estimates on the control cost in all times.
We will now apply them to the case $X = U = L^2 (\Omega)$, $A = H_\Omega = -\Delta + V$, and $B = \Eins_{S \cap \Omega}$. In this setting we have in particular $\beta = 0$, and the spectral inequality reads
\begin{equation} \label{eq:uncertainty}
\forall \lambda > \infspec \ \forall \phi \in L^2(\Omega) \colon \quad
\lVert P_{H_\Omega}(\lambda) \phi \rVert_{L^2 (\Omega)}^2
  \leq
  d_0 \euler^{d_1 (\lambda - \infspec)^\gamma}
  \lVert \Eins_{S \cap \Omega} P_{H_\Omega}(\lambda) \phi \rVert_{L^2 (\Omega)}^2.
\end{equation}
Recall that $\infspec$ denotes the infimum of the spectrum of the operator under consideration, i.e.~here $\infspec=\inf\sigma(H_\Omega)$.
We start by defining two geometric situations for the subset $S \subset \Omega$ where \eqref{eq:uncertainty} is satisfied, cf.\ Fig.~\ref{fig:sets}. For a measurable set $M \subset \RR^d$ we denote by $\lvert M \rvert$ its Lebesgue measure, and for $x \in \RR^d$ and $\rho > 0$ we denote by $B (x , r) = \{y \in \RR^d \colon \lvert x-y \rvert < r\}$ the ball of radius $r$ centered at $x$.
\begin{definition}[Equidistributed set]
 Let $G,\delta> 0$. We say that a set $S \subset \RR^d$ is \emph{$(G,\delta)$-equidistributed} if $S$ is measurable, and
 \[
  \forall  j \in (G\ZZ)^d\ \exists z_j \in \Lambda_G + j \colon
  \quad
  B (z_j , \delta) \subset S \cap (\Lambda_G + j) .
 \]
\end{definition}
\begin{figure}[ht]\centering
\begin{tikzpicture}
\begin{scope}
\clip (-0.2,-0.2) rectangle (5.2,5.2);
\pgfmathsetseed{{\number\pdfrandomseed}}
\draw[dotted] (-.2, -.2) grid (5.2,5.2);
\foreach \x in {0.5,1.5,...,4.5}{
  \foreach \y in {0.5,1.5,...,4.5}{
  	\pgfmathrandominteger{\radius}{1}{10}
    \filldraw[fill=gray!70] (\x+rand*0.35,\y+rand*0.35) circle (0.15+0.02*\radius);
  }
}
\end{scope}
\begin{scope}[xshift=6cm]
  \pgfmathsetmacro{\X}{7};
  \pgfmathsetmacro{\Y}{4};
  \pgfmathsetmacro{\e}{4};
  \pgfmathsetmacro{\d}{4};
  \foreach \y in {0,...,\Y}{
	\foreach \x in {0,...,\X}{
		\pgfmathsetmacro{\A}{random(1,\e)}; % choose number of columns
		\pgfmathsetmacro{\a}{\A^2-1};
                \pgfmathsetmacro{\B}{random(1,\d)};  % choose number of rows
		\pgfmathsetmacro{\b}{\B^2-1};
		\pgfmathsetmacro{\Lo}{1/\A^2}; % horizontal step length
		\pgfmathsetmacro{\Lv}{1/\B^2}; % vertical step length

% 		% draw one from the building blocks

		\foreach \i in {0,..., \a}{
			  %set coordinate in x direction of the cubes to color /not to color
			  \pgfmathsetmacro{\v}{\x +(\i *\Lo)};
			  \pgfmathsetmacro{\V}{\x +(\i +1)*\Lo};
			  \foreach \j in {0,...,\b}{
				% set coordinate in y direction of the cubes to color /not to color
				\pgfmathsetmacro{\W}{\y + (\j +1)*\Lv};
				\pgfmathsetmacro{\w}{\y +(\j *\Lv)};
						% set parity to check
						\pgfmathsetmacro{\test}{\j +\i};
						% check parity to decide to /not to color
						\ifthenelse{\isodd{\test}}{
						}{\filldraw[black!40] (\v,\w) rectangle (\V, \W);}
			}
		}
		\draw (\x,\y) rectangle (\x+1, \y+1);
	}
  }
\end{scope}
\end{tikzpicture}
\caption{Illustrations of an equidistributed set (left) and a thick set (right).\label{fig:sets}}
\end{figure}
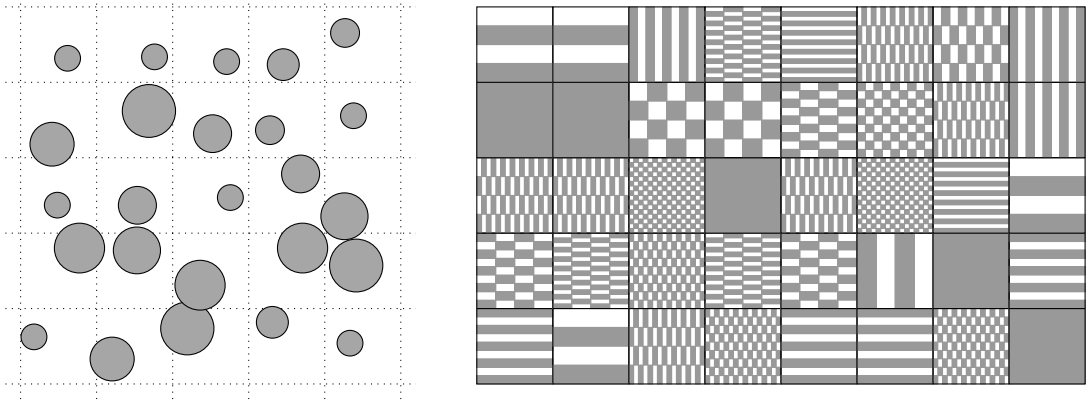
\begin{definition}[Thick set]
 Let $\rho \in (0,1]$ and $a = (a_1 , \ldots , a_d) \in \RR^d$ with $a_j > 0$ for $j \in \{1,\ldots , d\}$. We say that a set $S\subset \RR^d$ is \emph{$(\rho , a)$-thick} if $S$ is measurable, and for each parallelepiped
 \[
 P = \bigtimes_{j=1}^d \left[x_j - \frac{a_j}{2} , x_j + \frac{a_j}{2} \right]  \quad \text{with} \quad
 x_j \in \RR
 \quad \text{for} \quad
 j \in \{1,\ldots , d\}
 \]
 we have
 \[
  \left\lvert S \cap P \right\rvert \geq \rho \left\lvert P \right\rvert .
 \]
\end{definition}
Note that every $(G,\delta)$-equidistributed set is $(\rho , a)$-thick for some $\rho$ and $a$ but there exist $(\rho , a)$-thick sets which are not $(G,\delta)$-equidistributed for any $G$ and $\delta$.
\par
Now, we cite three spectral inequalities, i.e.\ uncertainty relations for spectral projectors.
\begin{theorem}[\cite{NakicTTV-18,NakicTTV-18+}]
  \label{thm:UCP_NTTV}
Let $G,\delta>0$,
% $\alpha_i - \beta_i \geq G$,
$\Omega$ as in \eqref{eq:Omega} with
$\Lambda_G \subset \Omega$, $S \subset \RR^d$ be $(\delta , G)$-equidistributed, $V \in L^\infty (\Omega)$ real-valued, and $\lambda \in \RR$. Then we have for all $\phi \in L^2(\Omega)$
\begin{equation*}
\lVert P_{H_\Omega}(\lambda) \phi \rVert_{L^2 (S \cap \Omega)}^2
\geq C_\si \lVert P_{H_\Omega}(\lambda) \phi \rVert_{L^2 (\Omega)}^2 ,
\]
where
\[
 C_\si = \sup_{\zeta \in \RR}
 \left(\frac{\delta}{G}\right)^{N \bigl(1 + G^{4/3}\lVert V-\zeta \rVert_\infty^{2/3} + G\sqrt{(\lambda-\zeta)_+} \bigr)} ,
\end{equation*}
$t_+:=\max\{0,t\} $ for $t \in \RR$, and where $N > 0$ is a constant depending only on the dimension. In particular, we have for all $\lambda \geq \infspec$
\[
C_\si \geq d_0 \euler^{d_1 (\lambda - \infspec)^{1/2}}
\quad\text{with}\quad
 d_0 = \left(\frac{\delta}{G}\right)^{N (1 + G^{4/3} \lVert V - \infspec \rVert_\infty^{2/3})}
 \quad\text{and}\quad
 d_1 = N G \ln \left(\frac{\delta}{G} \right).
\]
\end{theorem}
The following result was proven in the $\RR^d$ case in \cite{Kovrijkine-00,Kovrijkine-01} and adapted to cubes and some other geometries in \cite{EgidiV-17-arxiv,EgidiV-18,Egidi-18}.
Such estimates are often called Logvinenko-Sereda Theorems. We do not expound the history of this topic but refer the reader
e.g.\ to the survey \cite{EgidiNSTTV-20}.
\begin{theorem}[{\cite{Kovrijkine-00,EgidiV-17-arxiv,EgidiV-18}}]
\label{thm:Lognivenko-Sereda}
 Let $V = 0$, and $\Omega = \RR^d$ or $\Omega = \Lambda_L$ for some $L > 0$. Let further $S \subset \RR^d$ be a $(\rho , a)$-thick set. If $\Omega = \Lambda_L$ we assume that $a_j \leq L$ for all $j \in \{1,\ldots , d\}$.
 Then we have for all $\phi \in L^2(\Omega)$
 \[
 \lVert P_{H_\Omega}(\lambda) \phi \rVert_{L^2 (\Omega)}^2 \leq d_0 \euler^{d_1 \lambda^{1/2}}  \lVert P_{H_\Omega}(\lambda) \phi \rVert_{L^2 (S \cap \Omega)}^2,
 \]
 where
 \[
 d_0 = \left( \frac{C^d}{\rho}\right)^{Cd}
 \quad\text{and}\quad
 d_1 =  C \lvert a \rvert_1 \ln \left( \frac{C^d}{\rho} \right).
 \]
 Here, $C$ is a universal positive constant.
\end{theorem}
Note that the domain $\Omega$ itself does not enter the constant in the upper bounds neither in Theorem~\ref{thm:UCP_NTTV} nor in Theorem~\ref{thm:Lognivenko-Sereda}.
For this reason we call such  an uncertainty relation \emph{scale-free}.
\begin{remark}
In the case where $\Omega = \RR^d$ Theorem~\ref{thm:Lognivenko-Sereda} has been proven in \cite{Kovrijkine-00} under the assumption that the Fourier transform $\mathcal{F} \phi$ of $\phi$ satisfies
\begin{equation} \label{eq:Fourier_Kov}
  \supp (\mathcal{F} \phi) \subset \bigtimes_{j=1}^d \left[x_j - \frac{b_j}{2} , x_j + \frac{b_j}{2} \right]  \quad \text{for some} \quad
 x_j \in \RR \ \text{and} \ b_j > 0,
 \quad j \in \{1,\ldots , d\} .
 \end{equation}
 Here, $\cF$ denotes the standard Fourier transformation on $L^2 (\RR^d)$. In the case where $\Omega = \Lambda_L$ and periodic boundary conditions. Theorem~\ref{thm:Lognivenko-Sereda} has been proven in \cite{EgidiV-17-arxiv} under assumption \eqref{eq:Fourier_Kov}.
 Here, the Fourier transform $\mathcal{F} \phi$ of $\phi \in L^2 (\Lambda_L)$ is given by
 \begin{equation*}
  \mathcal{F} \phi : \left( \frac{2\pi}{L} \ZZ \right)^d \to \CC, \quad
  (\mathcal{F} \phi) (k)
  =
  \frac{1}{L^d} \int_{\Lambda_L} \phi(x) \euler^{-\mathrm{i} (x \cdot k)} \drm x .
 \end{equation*}
 In both cases, one can show that functions $\phi \in \ran P_{H_\Omega}( \lambda )$ as considered in Theorem~\ref{thm:Lognivenko-Sereda} satisfy Assumption~\eqref{eq:Fourier_Kov} with $x_j = 0$, and $b_j = 2\sqrt{\lambda}$. This has been carried out in Section~5 of \cite{EgidiV-18}. This statement remains true if $\Omega = \Lambda_L$ and Dirichlet or Neumann boundary conditions are imposed, see again Section~5 of \cite{EgidiV-18}.
\end{remark}
 Theorem~\ref{thm:Lognivenko-Sereda} immediately implies the following
 \begin{theorem}
 %\label{thm:fractional_thick}
  In the situation of Theorem~\ref{thm:Lognivenko-Sereda}, we have for all $\lambda \geq 0$, all $\theta > 0$ and all $\phi \in L^2(\Omega)$ that
  \begin{equation}
    \label{eq:uncertainty_relation_fractional_Laplacian}
    \lVert P_{(-\Delta)^\theta}(\lambda) \phi \rVert_{L^2(\Omega)}^2
    \leq
    d_0 \euler^{d_1 \lambda^{1/(2 \theta)}}
    \lVert P_{(-\Delta)^\theta}(\lambda) \phi \rVert_{L^2(S \cap \Omega)}^2
  \end{equation}
  where
  \[
  d_0 = \left( \frac{C^d}{\rho} \right)^{Cd}
  \quad\text{and}\quad
  d_1 =  C \lvert a \rvert_1 \ln \left( \frac{C^d}{\rho} \right),
  \]
 and $C$ is a universal positive constant.
 \end{theorem}
 \begin{proof}

 We estimate, using the transformation formula for spectral measures, cf.\ \cite[Prop.~4.24]{Schmuedgen-12}, and Theorem~\ref{thm:Lognivenko-Sereda}
  \begin{align*}
  \lVert P_{(- \Delta)^\theta}(\lambda) \phi \rVert_{L^2(\Omega)}^2
  =
  \lVert  P_{(- \Delta)}(\lambda^{1/\theta}) \phi \rVert_{L^2(\Omega)}^2
  &\leq
  d_0 \euler^{d_1 \lambda^{1/(2 \theta)}}
  \lVert P_{(- \Delta)} (\lambda^{1/\theta}) \phi \rVert_{L^2(S \cap \Omega)}^2\\
  &=
  d_0 \euler^{d_1 \lambda^{1/(2 \theta)}}
  \lVert P_{(- \Delta)^\theta} (\lambda) \phi \rVert_{L^2(S \cap \Omega)}^2
  .
  \qedhere
 \end{align*}
\end{proof}
Note that the exponent $1/2\theta$ in \eqref{eq:uncertainty_relation_fractional_Laplacian} is smaller than one if $\theta >1/2$.
\begin{remark}\label{rem:LMoyano}
In the recent preprint~\cite{LebeauM-19}, G. Lebeau and I. Moyano prove spectral inequalities for regular Schr\"odinger operators with thick observation sets.
More precisely, they consider Schr\"odinger operators $H_\Omega=-\Delta+V$,
where $V\colon \Omega:=\RR^d \to \RR$ can be extended to a holomorphic function
$V\colon \{z\in\CC^d: \lvert \Im(z) \rvert < \epsilon \} \to \RR$ for some $\epsilon\in(0,1)$, which decays at infinity, i.e.\ $\lim_{\lVert x \rVert \to\infty} V(x) =0$,
and satisfies for $\alpha\in \NN_0^d$, $\lvert\alpha \rvert \leq 2$ and $z\in\CC^d$ with $ \lvert \Im(z) \rvert < \epsilon$
\[
 \lvert \partial^\alpha (z) \rvert \leq C (1 + \lvert z \rvert )^{-\epsilon- \lvert \alpha \rvert}.
\]
(In fact, \cite{LebeauM-19} allows for $\Delta$ to be a Laplace-Beltrami operator with certain analytic metrics on $\RR^d$.)
\end{remark}
\begin{theorem}[\cite{LebeauM-19}]
\label{thm:Lebeau-Moyano}
 Let $S \subset \RR^d$ be a $(\rho , a)$-thick set and $H_\Omega$ as in Remark \ref{rem:LMoyano}.
 Then there are  constants $d_0, d_1\geq0$ which depend on the properties of the potential $V$ and the thick set $S$, such that for all $\phi \in L^2(\RR^d)$ one has
\begin{equation*}
 %\label{eq:LM-spectral-inequality}
 \Vert P_{H_\Omega}(\lambda)\phi \Vert^2_{L^2(\RR^d)}
 \leq d_0 \euler^{d_1\sqrt{\lambda_+}}
 \Vert P_{H_\Omega}(\lambda)\phi \Vert^2_{L^2(S)}.
\end{equation*}
\end{theorem}

Now we combine the spectral inequalities from the previous theorems in this section
with Theorem~\ref{thm:obs_and_control} and immediately deduce the following explicit estimates on the control cost.
\begin{theorem}[Negative Laplacian with control on thick sets]\label{thm:thick}
 Let $\Omega = \RR^d$, or $\Omega = \Lambda_L$ for some $L > 0$. Let further $S \subset \RR^d$ be a $(\rho , a)$-thick set. If $\Omega = \Lambda_L$ we assume that $a_j \leq L$ for all $j \in \{1,\ldots , d\}$. Then for all $\phi \in L^2 (\Omega)$, and all $T > 0$ we have
 \[
  \lVert \euler^{\Delta T} \phi \rVert_{L^2 (\Omega)}^2
  \leq
  \const^2
  \int_0^T \lVert \euler^{\Delta t} \phi \rVert_{L^2 (S \cap \Omega)}^2 \drm t ,
  \]
  where
  \[
  \const^2
 =
 \frac{C_1^{d^2}}{T} \rho^{-C_2d}
   \exp \left( \frac{C_3 \lvert a \rvert_1^2 \ln^2 (C_4^d / \rho)}{T} \right).
 \]
 Here, $C_1$, $C_2$, $C_3$, and $C_4$ are universal positive constants.
 Moreover, for all $T > 0$ the system \eqref{eq:heat_equation} with $V = 0$ is null-controllable in time $T$, and the cost satisfies $\cost \leq \const$.
\end{theorem}

\begin{theorem}[Fractional negative Laplacian with control on thick sets]\label{thm:thick_fractional}
 Let $\Omega = \RR^d$, or $\Omega = \Lambda_L$ for some $L > 0$ and let $\theta > 1/2$.
 Let further $S \subset \RR^d$ be a $(\rho , a)$-thick set. If $\Omega = \Lambda_L$ we assume that $a_j \leq L$ for all $j \in \{1,\ldots , d\}$. Then for all $\phi \in L^2 (\Omega)$, and all $T > 0$ we have
 \[
  \lVert \euler^{- (-\Delta)^\theta T} \phi \rVert_{L^2 (\Omega)}^2
  \leq
  \const^2
  \int_0^T \lVert \euler^{- (-\Delta)^\theta t} \phi \rVert_{L^2 (S \cap \Omega)}^2 \drm t ,
  \]
  where
  \[
  \const^2
 =
 \frac{C_1}{T} \rho^{-C_2 d} \exp \left( \frac{C_3 \big( \lvert a \rvert_1 \ln( C_4^d / \rho) \big)^{\frac{2 \theta}{2 \theta - 1}}}{T^{\frac{1}{2 \theta - 1}}} \right) .
 \]
 Here, $C_1$, $C_2$, $C_3$, and $C_4$ are universal positive constants.
 Moreover, for all $T > 0$ the system
\begin{equation} \label{eq:heat_equation_fractional}
\dot w  + (-\Delta)^\theta w = \Eins_{S\cap \Omega} u ,
 \quad
 w(0,\cdot) = w_0 \in L^2(\Omega),
 \end{equation}
is null-controllable in time $T$, and the cost satisfies $\cost \leq \const$.
\end{theorem}
\begin{theorem}[Schr\"odinger operator with control on equidistributed sets] \label{thm:equi}
 Let $G,\delta > 0$, $\Omega \subset \RR^d$ be as in \eqref{eq:Omega} with
%  $\alpha_i - \beta_i \geq G$ for all $i \in \{ 1, \dots, d \}$,
 $\Lambda_G \subset \Omega$, $S \subset \RR^d$ be $(\delta , G)$-equidistributed, and $V \in L^\infty (\Omega)$ real-valued.
 Then for all $\phi \in L^2 (\Omega)$, and all $T > 0$ we have
 \[
  \lVert \euler^{- H_\Omega T} \phi \rVert_{L^2 (\Omega)}^2
  \leq
  \const^2
  \int_0^T \lVert \euler^{- H_\Omega t} \phi \rVert_{L^2 (S \cap \Omega)}^2 \drm t ,
 \]
 where
\begin{align*}
\const^2
 &= \left( \frac{\delta}{G} \right)^{-C_2 (1 + G^{4/3} \lVert V - \infspec \rVert_\infty^{2/3})} \inf_{t \in (0,T]}
      \frac{C_1}{t}
      \exp \left( \frac{C_3 G^2 \ln^2 (\delta / G)}{t} - 2 \infspec t \right) & & \text{if $\infspec < 0$}, \\[2ex]
 \const^2
 &=  \left( \frac{\delta}{G} \right)^{-C_2 (1 + G^{4/3} \lVert V \rVert_\infty^{2/3})} \frac{C_1}{T} \exp \left( \frac{C_3 G^2 \ln^2 (\delta / G)}{T} \right) & & \text{if $\infspec \geq 0$},
  \\[2ex]
\const^2 &=  \left( \frac{\delta}{G} \right)^{-C_2 (1 + G^{4/3} \lVert V - \infspec \rVert_\infty^{2/3})} \inf_{t \in [0,T)}
      \frac{C_1}{T-t}
      \exp \left( \frac{C_3 G^2 \ln^2 (\delta / G)}{T-t} - 2 \infspec t \right) & & \text{if $\infspec > 0$}.
\end{align*}
Here, $C_1$, $C_2$, and $C_3$ are positive constants depending only on the dimension.
Moreover, for all $T > 0$ the system \eqref{eq:heat_equation} is null-controllable in time $T$, and the cost satisfies $\cost \leq \const$.
\end{theorem}
\begin{remark}
Note that the Theorems~\ref{thm:thick}, \ref{thm:thick_fractional} and \ref{thm:equi}
inherit the property to be scale-free, i.e.\ the constants in the upper bounds  are uniform in $\Omega \subset \RR^d$. They depend on $V$ only via its $L^\infty$-norm.
This will be relevant in certain asymptotic regimes, cf.~Remark~\ref{rem:Miller}.
\end{remark}
Theorem~\ref{thm:Lebeau-Moyano} of Lebeau and Moyano enables us to derive observability and control cost estimates for
controls on thick sets and Schr\"odinger semigroups with analytic potentials, again using Theorem~\ref{thm:obs_and_control}.
However, in this case the explicit dependence on the geometric properties of the thick  set is not exhibited.
\begin{theorem}[Regular Schr\"odinger operator with control on thick sets]
%\label{cor:Lebeau-Moyano}
 Let $S \subset \RR^d$ be a $(\rho , a)$-thick set, $H_\Omega$ as in Remark \ref{rem:LMoyano}, and $\infspec=\min \sigma (H_\Omega)\geq0$.
Then for all $\phi \in L^2 (\RR^d)$, and all $T > 0$ we have
 \begin{equation*}
  \lVert \euler^{- H_\Omega T} \phi \rVert_{L^2 (\RR^d)}^2
  \leq
  \const^2
  \int_0^T \lVert \euler^{- H_\Omega t} \phi \rVert_{L^2 (S)}^2 \drm t ,
\end{equation*}
where
\begin{equation*}
  C_\obs^2
  \leq
  \frac{C_1 d_0}{T} (2 d_0   + 1)^{C_2}
  \exp
  \left(
    C_3
    \frac{d_1^2}{T}
  \right) .
 \end{equation*}
Here $C_1$, $C_2$, $C_3$, $d_0$, and $d_1$ are the constants coming from Theorems~\ref{thm:obs_and_control} and \ref{thm:Lebeau-Moyano}.
Moreover, for all $T > 0$ the system \eqref{eq:heat_equation} is null-controllable in time $T$, and the cost satisfies $\cost \leq \const$.
 \end{theorem}

 \begin{remark}
 \label{rem:Miller}
In Theorems~\ref{thm:thick},~\ref{thm:thick_fractional}, and~\ref{thm:equi}, the asymptotic behavior of the upper bound on $\cost$ in the limit $T \to 0$ and $T \to \infty$ is optimal as discussed in Remark~\ref{rem:table}, see also Table~\ref{table:asymptotics}.
We also note that the term $\lVert V \rVert_\infty^{2/3}$ in the bound in Theorem~\ref{thm:equi} is optimal, at least in even space dimensions, see~\cite{DuyckaertsZZ-08}.
\par
 Furthermore, let us emphasize that the dependence of the rate of the exponential term on the parameter $G$ in Theorem~\ref{thm:equi} is optimal.
In order to illustrate this we will choose a geometry to which our upper bounds on the control cost apply  as well as the lower bound established by Miller in~\cite{Miller-04}.
Consider a hypercube of side length $L$
and for any $G$ dividing $L$,
let $S\cap \Lambda_L$ be the union of right halfs of elementary $G$-cells in the $L$-torus, see Figure~\ref{fig:right_halfs}.
The set $S\cap \Lambda_L$ is the restriction of a $(G, G/4)$-equidistributed set $S$ to $\Lambda_L$.
\begin{figure}[ht] \centering

    \begin{tikzpicture}[scale=1]

    \def\n{1} \def\f{1/2}

     \foreach \y in {0,...,3}{
    \foreach \x in {0,...,3}{
        \filldraw[black!20] (\x+.5,\y) rectangle (\x+1,\y+1);
	\draw[step=1,black!50,very thin] ((\x,\y) rectangle (\x+.5,\y+1);
        }}
	\draw[very thick] (0,0) rectangle (4,4);
    \draw[step=1,black!50,very thin] (0,0) grid (4,4);
    \draw[very thick] (0,0) rectangle (4,4);
    \begin{scope}[xshift = 5cm]
     \draw[very thick, dotted] (0,0) rectangle (1,1);
     \filldraw[black!20] (0,2) rectangle (.5,3);
     \draw[anchor = west] (1.5,.5) node {$\Lambda_G$};
     \draw[anchor = west] (1,2.5) node {$S$};
    \end{scope}
     \draw (-.5, 2) node {$\Lambda_L$};
  \end{tikzpicture}
  \caption{The $(G,\delta)$-equidistributed set $S \subset \Lambda_L$ in Remark~\ref{rem:Miller}}
  \label{fig:right_halfs}
 \end{figure}
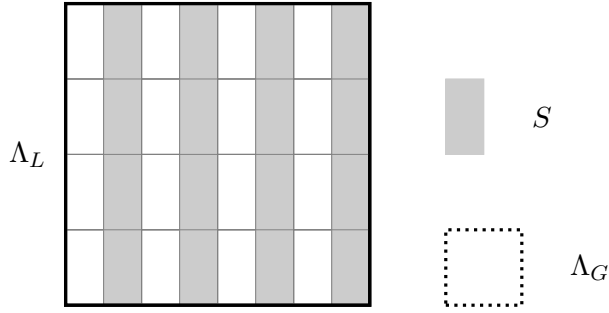
For $H_\Omega=-\Delta$ with periodic boundary conditions on $\Omega=\Lambda_L$  Theorem~\ref{thm:equi} implies
 \begin{equation}
  \label{eq:upper_bound_geometry}
  \limsup_{T \to 0} T \ln C_T
  \leq
  C G^2
 \end{equation}
with a positive constant $C$. Note that due to periodic boundary conditions $H_\Omega$ can be considered as minus times the Laplacians on the torus.
Thus the  following result of Miller~\cite{Miller-04} applies: On bounded, smooth, and connected manifolds $\Omega$ with control operator $B = \chi_S$ for an open $S \subset \Omega$ the control cost $C_T$ in time $T$ satisfies
 \begin{equation}
  \label{eq:lower_bound_geometry_Miller}
  \sup_{y \in \Omega} \operatorname{dist} (y,\overline S)^2 / 4
  \leq
  \liminf_{T \to 0} T \ln C_T.
 \end{equation}
For every $G$ dividing $L$, the left hand side of~\eqref{eq:lower_bound_geometry_Miller} is precisely $G^2/64$.
 Combining~\eqref{eq:lower_bound_geometry_Miller} and~\eqref{eq:upper_bound_geometry} gives
 \begin{equation}
  \label{eq:Miller_comparing_bounds}
\frac{G^2}{64}  \leq   \liminf_{T \to 0} T \ln C_T \leq  C \ G^2
  \quad
  \text{for all $G$ which divide $L$}.
 \end{equation}
Now we see that for any $\epsilon>0$ a bound of the type $\limsup_{T \to 0} T \ln C_T \leq C G^{2 -\epsilon}$ is impossible by considering the limit $G\to \infty$.
(Note that this requires $L\to \infty$. However, this is covered by our results,
since the control cost bounds are scale-free, i.e.\ $L$-independent.)
Similarly, the asymptotic behavior for  $G\to 0 $ excludes a bound of the type $\limsup_{T \to 0} T \ln C_T \leq C G^{2 +\epsilon}$.
Since every $(G,\delta)$-equidistributed set is $(\rho,a)$-thick for some $\rho$ and $a$, the above example also applies to the upper bound in Theorem~\ref{thm:thick}.
\par
We also note that the restriction of $L$ being a multiple of $G$ in~\eqref{eq:Miller_comparing_bounds} is unnecessary: The lower bound in~\cite{Miller-04} is based on approximate $\delta$-functions as initial states
for the heat evolution.
 It can be generalized from compact manifolds to the case of $\RR^d$, see for instance~\cite{EgidiV-18, WangWZZ-17-arxiv} where similar ideas are being used whence we obtain the inequality in~\eqref{eq:Miller_comparing_bounds} for all $G > 0$.
\par
 Another way to interpret  \eqref{eq:Miller_comparing_bounds} is that in the exponential, the geometric parameters $G$ is in the same relation to the time $T$ as the order of space and time derivatives in the underlying heat equation,
 i.e.\ $\ln C_T\sim G^2/T$.
\end{remark}
\section{Homogenization and de-homogenization}
\label{sec:homogenization}
We now investigate the behavior of the control cost in two asymptotic geometric regimes:
Homogenization and de-homogenization. Although there is an intuitive relation to well known homogenization theory, we are pursuing here a different, and to our best knowledge, novel question:
How does the control cost change if we modify the control set in a specific geometric way?

For the purpose of this paper a homogenization limit of the control set means
that we consider a sequence of operators $B_n, n \in \NN,$ each one the indicator function of a set
$S(n)\subset \RR^d$, with the property that $S(n)$ is $(\rho,a_n)$-thick, with $\rho>0$ independent of
$n$,  and $a_n\to 0$ for $n\to \infty$. De-homogenization is defined analogously, see below.

The bounds of Theorems~\ref{thm:thick} and~\ref{thm:equi} provide for the first time sufficiently precise
estimates in order to pursue such an investigation.
They in turn crucially rely on
the abstract control cost estimate in Theorem~\ref{thm:obs_and_control} as well as the spectral inequalities, proved in~\cite{NakicTTV-18,EgidiV-18} and spelled out in Theorems~\ref{thm:UCP_NTTV} and~\ref{thm:Lognivenko-Sereda} here.
Indeed, the fact that upper bounds on the control cost in earlier literature were
insufficient to carry out our intended program
was the main motivation for proving the upper bounds on the control cost in Section~\ref{sec:abstract}.

Let us emphasize that we are here spelling out the behavior of \emph{a scalar quantity of interest}, namely the control cost bound. A fuller understanding of the (de-)homogenization limit would be provided if one would study the convergence of the underlying operators and functions, which we do not carry out here, see also
Remark~\ref{rem:homogenization_conjecture}.

\subsection*{Homogenization}
We will first treat homogenization.
This means that the control set $S \subset \Omega$ becomes more and more evenly distributed over the space while keeping an overall lower bound on the relative density.
This corresponds to reducing local fluctuations in the density of the control set $S$.
In the case of $(G,\delta)$-equidistributed sets this corresponds to $G$ and $\delta$ simultaneously tending to $0$ while their ratio remains constant.
In the case of $(\rho,a)$-thick sets, this corresponds to $a$ tending to zero while $\rho$ is kept constant.
We refer to Figure~\ref{fig:homogenization} for an illustration in the case of $(G,\delta)$-equidistributed sets.
\begin{figure}[ht]
 \centering
 \begin{tikzpicture}[scale = 1.25]
 \begin{scope}[xshift = -4cm, scale = .75]
\pgfmathsetmacro{\d}{0.2};
\foreach \x in {1,...,4}{
                \foreach \y in {1,...,4}{
                                %determine random coordinates of center of delta-ball
                        \pgfmathsetmacro{\xx}{rand *0.3+0.5};
                        \pgfmathsetmacro{\yy}{rand *0.3+0.5};
                                \draw[black!80,fill = black!40, opacity = 1.0] (\x+\xx,\y+\yy) circle (0.2cm);                                                       }};
\end{scope}
\draw (0,2) node {\Large $\Rightarrow$};
\begin{scope}[xshift = 0cm, yshift = 12, scale = .25]
\pgfmathsetmacro{\d}{0.2};
\foreach \x in {1,...,12}{
                \foreach \y in {1,...,12}{
                                %determine random coordinates of center of delta-ball
                        \pgfmathsetmacro{\xx}{rand *0.3+0.5};
                        \pgfmathsetmacro{\yy}{rand *0.3+0.5};
                                \draw[black!80,fill = black!40, opacity = 1.0] (\x+\xx,\y+\yy) circle (0.2cm);                                                       }};
\end{scope}
\draw (3.75,2) node {\Large $\Rightarrow$};
 \begin{scope}[xshift = 4cm, yshift = 18, scale = .1]
\pgfmathsetmacro{\d}{0.2};
\foreach \x in {1,...,30}{
                \foreach \y in {1,...,30}{
                        \pgfmathsetmacro{\xx}{rand *0.3+0.5};
                        \pgfmathsetmacro{\yy}{rand *0.3+0.5};
                                \draw[black!80,fill = black!40, opacity = 1.0] (\x+\xx,\y+\yy) circle (0.2cm);                                                        }};
\end{scope}
\end{tikzpicture}
\caption{Illustration of homogenization in the case of $(G, \delta)$-equidistributed sets}
\label{fig:homogenization}
\end{figure}
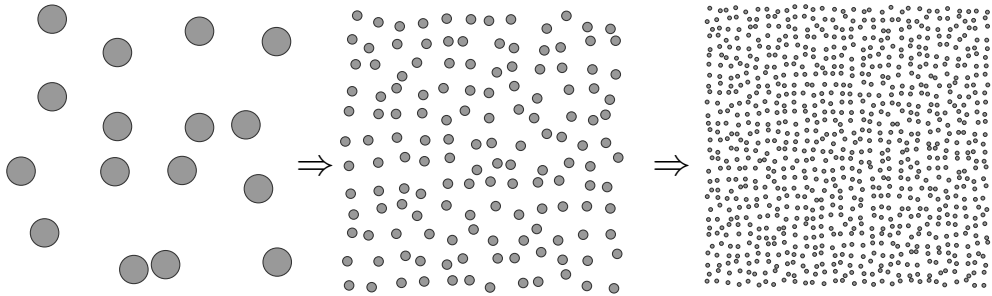
The first example shows that homogenization counteracts the exponential singularity of the control cost in the small time regime.
\begin{example}
 \label{ex:heat_pure}
 We consider the controlled heat equation~\eqref{eq:heat_equation} and assume that $V = 0$ and $\Omega = \RR^d$ or $\Omega = \Lambda_L$.
 We fix a density $\rho \in (0,1)$.
\par
 For every $a \in \RR^d$ with positive entries (and with $a_i \leq L$ if $\Omega = \Lambda_L$), and every $(\rho,a)$-thick set $S_a$, Theorem~\ref{thm:thick} implies that the system~\eqref{eq:heat_equation} with $S = S_a$ is null-controllable in every time $T > 0$ with
 \begin{equation}\label{eq:upper-bound-before}
\cost^2
 \leq
 \frac{C_1}{T} \rho^{-C_2d}
   \exp \left( \frac{C_3 \lvert a \rvert_1^2 \ln^2 (C_4^d / \rho)}{T} \right).
 \end{equation}
The exponential blow-up for small times appearing in \eqref{eq:upper-bound-before}
 is sharp in the following sense:  For the controlled heat equation \eqref{eq:heat_equation} with open $S\subset \Omega$, $S \neq \Omega$ it is known that the control cost grows at least proportionally to $\exp(\mathrm{const.}/T)$ as $T$ tends to zero, see e.g.\ \cite{FernandezZ-00,Miller-04}.
 Now, recall that homogenization means that we let the vector $a$ tend to zero while keeping $\rho$ constant.
 In the homogenization limit $a \to 0$, the upper bound in \eqref{eq:upper-bound-before} tends to
 \begin{equation}
  \label{eq:control_cost_in_limit_homogenization_C_positive}
  \frac{C_1}{T}
  \rho^{-C_2 d}.
 \end{equation}
 We see that in the limit $a \to 0$, the exponential singularity at $T = 0$, which is characteristic for controlled heat equation, disappears.
 We conclude that homogenization counteracts this $\exp(C/T)$ singularity.
Moreover, we note that the control cost is always bounded from below by the cost of a system with full control, i.e.\ with $S= \RR^d$ in \eqref{eq:heat_equation} or
$B=Id$ in \eqref{eq:abstract_control_system}, respectively.
By Theorem~\ref{thm:lower_bound_control_cost} we conclude that $C_T^2 \geq 1/T$ for all $a \in \RR^d$ with positive entries.
Hence, the term $1/T$ cannot vanish in the homogenization limit.
Note that the expression in \eqref{eq:control_cost_in_limit_homogenization_C_positive} coincides (up to constants) with the control cost of the system with full control and $\infspec = 0$ considered in Remark~\ref{remark:full_control}.

%%%%%%%%%%%%%%%%%%%%%%%%%%%%%%%%%%%%%%%%%%%%%%%%%%
%%%%%%%%%%%%%%%%%%%%%%%%%%%%%%%%%%%%%%%%%%%%%%%%%%
% Finally, let us note that the above described phenomenology holds also in the $L^p$-space setting of \cite{GallaunST} described in Theorem \ref{thm:GallaunST-LpRd}, namely
% in the homogenization limit $a \to 0$, the observability estimate \eqref{eq:observability-Banach} becomes
% \begin{equation*}
% \frac{D_1 M^{16}}{T^{1/r}} \left( \frac{K^d}{\rho} \right)^{D_2}
%  \end{equation*}
%%%%%%%%%%%%%%%%%%%%%%%%%%%%%%%%%%%%%%%%%%%%%%%%%%
%%%%%%%%%%%%%%%%%%%%%%%%%%%%%%%%%%%%%%%%%%%%%%%%%%

\end{example}
The next example shows that also in the presence of a potential, homogenization annihilates the exponential singularity at small times.
Furthermore, the effect of potentials on the control cost disappears in the homogenization regime -- up to the effect of the potential on $\infspec$.
\begin{example}
 \label{ex:heat_with_V}
 We consider the controlled heat equation with bounded and real-valued potential $V$ as in~\eqref{eq:heat_equation}.
 Note that $\infspec \geq - \lVert V \rVert_\infty$.
 For all $G,\delta > 0$ and all $(G,\delta)$-equidistributed sets $S_{G,\delta}$ such that $\Lambda_G \subset \Omega$, Theorem~\ref{thm:equi} implies that the system~\eqref{eq:heat_equation} with $S = S_{G,\delta}$ is null-controllable in every time $T > 0$ with
\begin{align*}
\cost^2
 &\leq \left( \frac{\delta}{G} \right)^{-C_2 (1 + G^{4/3} \lVert V - \infspec \rVert_\infty^{2/3})} \inf_{t \in (0,T]}
      \frac{C_1}{t}
      \exp \left( \frac{C_3 G^2 \ln^2 (\delta / G)}{t} - 2 \infspec t \right) & & \text{if $\infspec < 0$}, \\[2ex]
 \cost^2
 &\leq  \left( \frac{\delta}{G} \right)^{-C_2 (1 + G^{4/3} \lVert V \rVert_\infty^{2/3})} \frac{C_1}{T} \exp \left( \frac{C_3 G^2 \ln^2 (\delta / G)}{T} \right) & & \text{if $\infspec = 0$}, \\[2ex]
\cost^2 &\leq  \left( \frac{\delta}{G} \right)^{-C_2 (1 + G^{4/3} \lVert V - \infspec \rVert_\infty^{2/3})} \inf_{t \in [0,T)}
      \frac{C_1}{T-t}
      \exp \left( \frac{C_3 G^2 \ln^2 (\delta / G)}{T-t} - 2 \infspec t \right) & & \text{if $\infspec > 0$}.
\end{align*}
 Homogenization now means sending $G$ and $\delta$ to zero while keeping $\varrho:=\delta / G$ constant.
 In this limit, the upper bounds tend to
 \[
  \varrho^{-C_2}
  \inf_{t \in (0,T]}
  \frac{C_1}{t} \exp(- 2 \infspec t)
  ,
  \quad
  \varrho^{-C_2}
  \frac{C_1}{T}
  ,
  \quad
  \text{and}
  \quad
  \varrho^{-C_2}
  \inf_{t \in [0,T)}
  \frac{C_1}{T - t} \exp(- 2 \infspec t),
 \]
%
 %\[
%  \left( \frac{\delta}{G} \right)^{-C_2}
%  \inf_{t \in (0,T]}
%  \frac{C_1}{t} \exp(- 2 \infspec t)
%  ,
%  \quad
%  \left( \frac{\delta}{G} \right)^{-C_2}
%  \frac{C_1}{T}
%  ,
%  \quad
%  \text{and}
%  \quad
%  \left( \frac{\delta}{G} \right)^{-C_2}
%  \inf_{t \in [0,T)}
%  \frac{C_1}{T - t} \exp(- 2 \infspec t),
% \]
corresponding to the cases $\infspec > 0$, $\infspec = 0$, and $\infspec > 0$, where we used monotonicity in order to interchange limits and infima.
 It is straightforward to see that we recover the upper bounds from Table~\ref{rem:table} with $d_1 = 0$.
 This shows that in the homogenization limit, the limit of the upper bounds on the control cost
exhibits the same behavior as
the control cost of the system with full control from Remark~\ref{remark:full_control}.
 \par
 Moreover we see that the influence of the potential $V$ on the control cost is annihilated up to the effect of the potential on $\infspec$.
\end{example}
%

%{\color{blue}
 \begin{remark} \label{rem:homogenization_conjecture}
  We emphasize that taking limits of the parameters $a$ and $\delta, G$ in Examples~\ref{ex:heat_pure} and~\ref{ex:heat_with_V} does neither require nor
  imply that there exist a limit of the operators $\mathbf{1}_{S_a}$ or $\mathbf{1}_{S_{G,\delta}}$, respectively.
  Furthermore, even if a limit of the operators existed, our inequality would not yield a limit of the control cost, but merely a bound on the limsup.

  However, there are situations where one could expect that in a homogenization limit, the optimal null-controls  as well as the corresponding solutions converge.
  For instance let us consider
  $\Omega =(0,1)^d$, $X = U = L^2 (\Omega)$, $A = H_\Omega = -\Delta$, and $B_n = \Eins_{S_n\cap \Omega}$ where
\[
S_n = \bigcup_{\stackrel{k_1, \ldots, k_d \in\{0,\ldots, 2^n-1\}}{k_1+ \ldots+ k_d \ \text{even}} }
\left( [0,2^{-n}]^d+ 2^{-n}(k_1, \ldots, k_d)^{\mathrm{T}} \right)
\]
as illustrated in Figure~\ref{fig:homogenized}.
\begin{figure}[ht] \centering \label{fig:homogenized}
  \begin{tikzpicture}[scale=3.5]

    \def\n{1} \def\f{1/2}
    \foreach \x/\y in {0/0,\f/\f}
    \filldraw[black!20] (\x,\y) rectangle (\x+\f,\y+\f);
    \draw[step=\f,black!50,very thin] (0,0) grid (1,1);
    \draw[very thick] (0,0) rectangle (1,1);

    \begin{scope}[xshift=1.2cm,scale=0.25]

     \foreach \y in {0,2}{
    \foreach \x in {0,2}{
        \filldraw[black!20] (\x,\y) rectangle (\x+1,\y+1) rectangle (2+\x,2+\y);}}
    \draw[step=1,black!50,very thin] (0,0) grid (4,4);
    \draw[very thick] (0,0) rectangle (4,4);
    \end{scope}

    \begin{scope}[xshift=2.4cm,scale=1/8]

     \foreach \y in {0,2,4,6}{
    \foreach \x in {0,2,4,6}{
        \filldraw[black!20] (\x,\y) rectangle (\x+1,\y+1) rectangle (2+\x,2+\y);}}
    \draw[step=1,black!50,very thin] (0,0) grid (8,8);
    \draw[very thick] (0,0) rectangle (8,8);
    \end{scope}
  \end{tikzpicture}
\caption{Illustration of the sets $B_1$, $B_2$, and $B_3$ in space dimension $d=2$}
\end{figure}
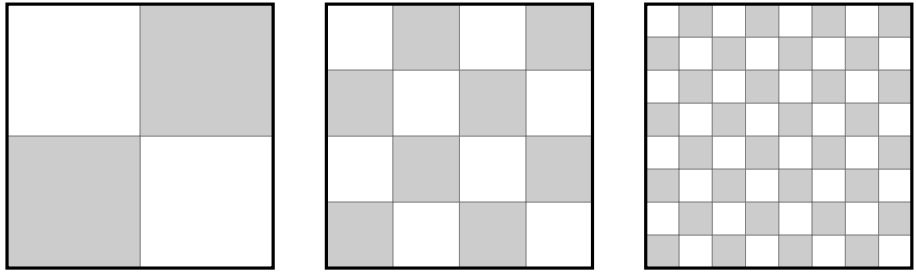
The corresponding heat equation
 \begin{equation*}
 \dot w - \Delta w = \Eins_{S_n} u,
 \quad
 w(0) = w_0 \in L^2(\Omega) ,
 \end{equation*}
has a unique null-control $u_n$ with minimal norm. Denote the corresponding solution by $w_n$.
 In this situation we expect that the sequences will converge to the minimal null-control $u$ and solution $w$, respectively, of the equation
 \begin{equation*}
 \dot w - \Delta w = \frac{1}{2} u,
 \quad
 w(0) = w_0 \in L^2(\Omega) .
 \end{equation*}
 Related, but different approximation scenarios have been studied in \cite{SeelmannV-19},
 however under the assumption that the sequence $(B_n)_n$ converges strongly to some limiting operator $B$.
 These results cannot be applied in the situation spelled out above, among others since we do not have $\Eins_{S_n} \to (1/2) \Eins_{\Omega} $ in the strong topology.
 A detailed study of these questions we leave for another occasion.
 %}
 \end{remark}

\subsection*{De-homogenization}

Now we treat the complementary regime which we call de-homogenization.
Let $\Omega = \RR^d$.
For $(G,\delta)$-equidistributed sets, de-homogenization means sending $G$ and $\delta$ simultaneously to $\infty$ while $G/\delta$ remains constant.
In the context of $(\rho,a)$-thick sets this means that all coordinates of $a$ tend to $\infty$ while $\rho$ is constant.
\par
Even though the overall relative density of the control set remains, de-homogenization allows for larger and larger void areas between the components of $S$ where no control is applied, see Figure~\ref{fig:de-homogenization}.
 It is unsurprising that for fixed time $T$, our upper bound on the control cost will in general increase since the diffusive nature of the heat equation makes it harder for components of the control set to exert control in larger and larger areas where there is no or only little control.
In particular, the considerations in~\cite{Miller-04}, see also Remark~\ref{rem:Miller}, show that for fixed time, the control cost estimate must tend to $\infty$ in the de-homogenization limit.
\begin{figure}[ht]
 \centering
    \begin{tikzpicture}[scale = -1.25]

 \begin{scope}[xshift = -4cm, scale = .75]
\pgfmathsetmacro{\d}{0.2};
\foreach \x in {1,...,4}{
                \foreach \y in {1,...,4}{
                        \pgfmathsetmacro{\xx}{rand *0.3+0.5};
                        \pgfmathsetmacro{\yy}{rand *0.3+0.5};
                                \draw[black!80,fill = black!40, opacity = 1.0] (\x+\xx,\y+\yy) circle (0.2cm);                                                       }};
\end{scope}

\draw (0,2) node {\Large $\Rightarrow$};

 \begin{scope}[xshift = 0cm, yshift = 12, scale = .25]
\pgfmathsetmacro{\d}{0.2};

\foreach \x in {1,...,12}{
                \foreach \y in {1,...,12}{
                        \pgfmathsetmacro{\xx}{rand *0.3+0.5};
                        \pgfmathsetmacro{\yy}{rand *0.3+0.5};
                                \draw[black!80,fill = black!40, opacity = 1.0] (\x+\xx,\y+\yy) circle (0.2cm);                                                       }};
\end{scope}
\draw (3.75,2) node {\Large $\Rightarrow$};
 \begin{scope}[xshift = 4cm, yshift = 18, scale = .1]
\pgfmathsetmacro{\d}{0.2};
\foreach \x in {1,...,30}{
                \foreach \y in {1,...,30}{
                        \pgfmathsetmacro{\xx}{rand *0.3+0.5};
                        \pgfmathsetmacro{\yy}{rand *0.3+0.5};
                                \draw[black!80,fill = black!40, opacity = 1.0] (\x+\xx,\y+\yy) circle (0.2cm);                                                        }};
\end{scope}
\end{tikzpicture}
 \caption{Illustration of de-homogenization in the case of $(G, \delta)$-equidistributed sets}
 \label{fig:de-homogenization}
\end{figure}
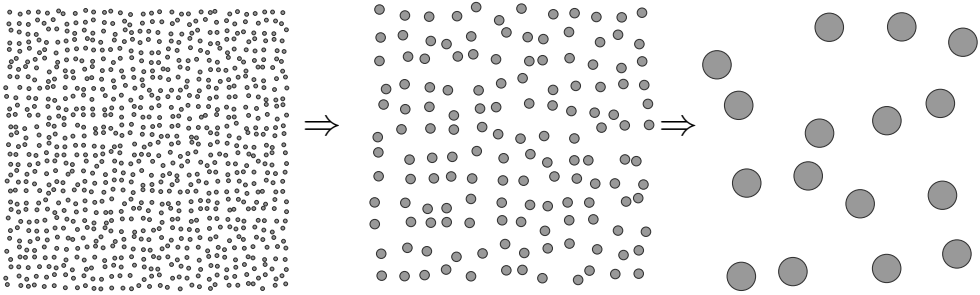
However, since $\cost$ is non-increasing in time, one can ask if it is possible to keep the control cost constant by simultaneously letting $T$ tend to $\infty$ in the de-homogenization regime.
The following example positively answers this question and provides a rate between the required time and the order of de-homogenization.
\begin{example}
 \label{ex:dehomogenization_V=0}
  We consider the controlled fractional heat equation~\eqref{eq:heat_equation_fractional} and assume that $\Omega = \RR^d$.
 We fix a density $\rho \in (0,1)$.
\par
 For every $T > 0$ and every $a \in \RR^d$ with positive entries, and every $(\rho,a)$-thick set $S_a$, Theorem~\ref{thm:thick_fractional} implies that the system~\eqref{eq:heat_equation_fractional} with $S = S_a$ is null-controllable in every time $T > 0$ with
 \[
  \cost^2
 \leq
 \frac{C_1}{T} \rho^{-C_2 d} \exp \left( \frac{C_3 \big( \lvert a \rvert_1 \ln( C_4^d / \rho) \big)^{\frac{2 \theta}{2 \theta - 1}}}{T^{\frac{1}{2 \theta - 1}}} \right).
\]
 There are three model parameters here: the parameters $\rho$ and $a$, describing the geometry of the control set, and the time $T$.
 Since we already chose $\rho$ and $a$, the only remaining way to accommodate for the increase in our upper bound when $a$ tends to infinity is to modify the remaining parameter $T$ by choosing
 \[
 T \sim \lvert a \rvert_1^{2 \theta}
 \]
 (due to the $1/T$ term in front of the exponential, we can even allow for a small logarithmic correction to this relation).
 We have recovered the relation between time and space derivatives $\dot w = -({- \Delta})^{\theta} w$ from the underlying fractional heat equation.
 This is an indication that our estimates on the control cost with respect to time and space parameters are close to being optimal.
\end{example}
\begin{example}
\label{ex:dehomogenization_V}
 We consider the controlled heat equation with bounded and real-valued potential $V$ as in~\eqref{eq:heat_equation}, and assume that $\Omega = \RR^d$ and $\infspec > 0$.
 For all $G,\delta > 0$ and all $(G,\delta)$-equidistributed sets $S_{G,\delta}$, Theorem~\ref{thm:equi} implies that the system~\eqref{eq:heat_equation} with $S_{G,\delta}$ is null-controllable in every time $T > 0$ with
\begin{align*}
 \cost^2
 &
 \leq
 \left( \frac{\delta}{G} \right)^{-C_2 (1 + G^{4/3} \lVert V - \infspec \rVert_\infty^{2/3})}
      \frac{2 C_1}{T}
      \exp \left( \frac{2 C_3 G^2 \ln^2 (\delta / G)}{T} - \infspec T \right) .
\end{align*}
As in Example~\ref{ex:dehomogenization_V=0}, the increase of the upper bound in the de-homogenization limit can be accommodated by choosing $T \sim G^{4/3}$.
Note that this exponent $G^{4/3}$ is related to the counterexample in~\cite{DuyckaertsZZ-08}.
However, there are special cases, such as a constant, positive potential $V$ in which choosing $T \sim G$ is sufficient to compensate the increase of the control cost in the de-homogenization regime.
\end{example}
In Example~\ref{ex:dehomogenization_V} we assumed $\infspec > 0$.
In the case where $\infspec \leq 0$ this argument does not work anymore.
Indeed, if $\infspec < 0$, we know from Theorem~\ref{thm:lower_bound_control_cost} and Corollary~\ref{cor:lower_bound_cost_control_cost} that $C_\infty = \inf_{T > 0} C_T > 0$ for every choice of $G$ and $\delta$.
It is an interesting question whether $C_\infty$ tends to infinity (and if yes at which rate) or remains bounded in the de-homogenization limit.
It seems that this is not accessible with the techniques presented in this paper.
\subsubsection*{Acknowledgments}
We are grateful to I.~Moyano for informing us about his joint work
with G.~Lebeau.
% We also thank the anonymous referee of \emph{Journal de Math\'ematiques Pures et Appliqu\'ees} for constructive comments in response to an earlier version of the manuscript.
We also thank the anonymous referees for constructive comments to an earlier version of the manuscript.
I.~V.~would like to thank Christoph Schumacher and Albrecht Seelmann for helpful discussions.
The initial phase of this research was supported by travel grants within the binational Croatian-German PPP-Project \emph{The cost of controlling the heat flow in a multiscale setting}.
I.~N.~was supported in part by the Croatian Science Foundation under the projects 9345 and IP-2016-06-2468.
The second named author was supported in part by the European Research Council starting grant 639305 (SPECTRUM).
%%
%\bibliographystyle{alpha}
%\bibliography{lit}
\newcommand{\etalchar}[1]{$^{#1}$}

\end{document}